\documentclass[12pt,final]{amsart} 
\usepackage{amsmath , amsthm , amsfonts , ifpdf}
\usepackage{color}
\usepackage{graphicx}
\usepackage{multicol}
\usepackage{epstopdf}
\usepackage{verbatim}

\usepackage{tikz}

\usepackage{amsmath}
\usepackage{amssymb}

\usetikzlibrary{arrows , shapes , trees , backgrounds}
\usetikzlibrary{intersections}

\theoremstyle{plain}
\newtheorem{theorem}{Theorem}[section]
\newtheorem*{theorem*}{Theorem}

\newtheorem{pro}[theorem]{Proposition}
\newtheorem{Def}[theorem]{Definition}
\newtheorem{lem}[theorem]{Lemma}

\newtheorem{cor}[theorem]{Corollary}

\theoremstyle{definition}
\newtheorem*{Def*}{Definition}

\newtheorem{Rem}[theorem]{Remark}

\numberwithin{equation}{section}

\newcommand{\bpo}{\begin{pro}}
\newcommand{\epo}{\end{pro}}
\newcommand{\be}{\begin{equation}}
\newcommand{\ene}{\end{equation}}
\newcommand{\br}{\begin{Rem}}
\newcommand{\er}{\end{Rem}}
\newcommand{\bl}{\begin{lem}}
\newcommand{\el}{\end{lem}}
\newcommand{\bd}{\begin{Def}}
\newcommand{\ed}{\end{Def}}
\newcommand{\ben}{\begin{enumerate}}
\newcommand{\een}{\end{enumerate}}
\newcommand{\bp}{\begin{proof}}
\newcommand{\ep}{\end{proof}}
\newcommand{\beq}{\begin{equation*}}
\newcommand{\eeq}{\end{equation*}}
\newcommand{\bear}{\begin{eqnarray*}}
\newcommand{\eear}{\end{eqnarray*}}
\newcommand{\bt}{\begin{theorem}}
\newcommand{\et}{\end{theorem}}
\newcommand{\bst}{\begin{split}}
\newcommand{\est}{\end{split}}

\newcommand{\bal}{\begin{aligned}}
\newcommand{\eal}{\end{aligned}}

\renewcommand{\P}{\partial}
\newcommand{\F}[2]{\frac{#1}{#2}}
\newcommand{\la}{\langle}
\newcommand{\ra}{\rangle}

\newcommand{\R}{\mathbb{R}}

\newcommand{\RM}{Riemannian manifold}

\newcommand{\wrt}{with respect to}
\newcommand{\Sc}{\varepsilon}
\renewcommand{\H}{\mathbb{H}}

\newcommand{\PLH}{{\mkern-1mu\times\mkern-1mu}}

\newcommand{\AF}{\mathfrak{F}}

\newcommand{\Ca}{Caccioppoli}
\newcommand{\tu}{\tilde{u}}

\newcommand{\bM}{\mathbb{M}}
\newcommand{\IMC}{integer multiplicity current}
\newcommand{\dk}{D_{k,\alpha}}
\newcommand{\bdk}{\bar{D}_{k,\alpha}}
\newcommand{\tka}{T_{k,\alpha}}
\newcommand{\tps}{\tilde{\phi}}
\newcommand{\ts}{\tilde{\sigma}}
\newcommand{\fG}{\mathcal{G}}

\def\XXint#1#2#3{{\setbox0=\hbox{$#1{#2#3}{\int}$}
		\vcenter{\hbox{$#2#3$}}\kern-.5\wd0}}

\makeatletter
\def\@citestyle{\m@th\upshape\mdseries}
\def\citeform#1{{\bfseries#1}}
\def\@cite#1#2{{%
		\@citestyle[\citeform{#1}\if@tempswa, #2\fi]}}
\@ifundefined{cite }{%
	\expandafter\let\csname cite \endcsname\cite
	\edef\cite{\@nx\protect\@xp\@nx\csname cite \endcsname}%
}{}
\makeatother
\begin{document}

\title[Area minimizng problems in a cone]{The area  minimizing problem in conformal cones} 
\author{Qiang Gao, Hengyu Zhou}
\address[Q.~G]{Department of Mathematics, Sun Yat-Sen University, 510275, Guangzhou, P. R. China}
\email{gaoqiangks@outlook.com}

\address[H. ~Z]{ College of Mathematics and Statistics, Chongqing University, Huxi Campus, Chongqing, 401331, P. R. China}
\address{Chongqing Key Laboratory of Analytic Mathematics and Applications, Chongqing University, Huxi Campus, Chongqing, 401331, P. R. China}
\email{zhouhyu@cqu.edu.cn}
\begin{abstract} In this paper we study the area minimizing problem in some kinds of conformal cones.  This concept is a generalization of the cones in Eulcidean spaces and the cylinders in product manifolds. We define a non-closed-minimal (NCM) condition for bounded domains. Under this assumption and other necessary conditions we establish the existence of bounded minimal graphs in mean convex conformal cones. Moreover those minimal graphs are the solutions to corresponding area minizing problems. We can solve the area minimizing problem in non-mean convex translating conformal cones if these cones are contained in a larger mean convex conformal cones with the NCM assumption.  We give examples to illustrate that this assumption can not be removed for our main results.
	\end{abstract}
\date{\today}
\maketitle
\section{Introduction}
In this paper we study the area minimizing problem in conformal cones.  A conformal cone is defined as follows. 
\begin{Def}\label{def:conformal:cone}  Suppose $\Omega$ is an open bounded Riemanian manifold with $C^2$ boundary and metric $\sigma$. Let $I$ be an open interval $(-\infty, A)$ where $A$ is a constant or $+\infty$. Let $\phi(r)$ be a $C^2$ positive function on $I$.  We call  $ (\Omega\PLH I,  \phi^2(r)(\sigma+dr^2)) 
	$ is a conformal cone, written as $Q_\phi$. \\
	\indent If $\phi(r)=e^{\F{\alpha}{n} r}$ for a constant $\alpha \in \R$, we call such cone as a translating conformal cone.  Let $C$ be an adjective. If $\Omega$ is $C$ (has the $C$ property), we call such cone as a $C$ conformal cone (with the $C$ property). 
\end{Def} 

\indent The above definition is a generalization of Euclidean cones $(\phi(r)=e^{2r},\Omega\subset S^{n+1})$, cylinders of product manifolds $(\phi(r)\equiv1)$ and a large class of warped product manifolds (remark \ref{condition:Remark:one}). Note that a conformal cone may be incomplete at the negative infinity.  \\
 \indent Throughout this paper let $\bar{Q}_\phi$ denote the set $\bar{\Omega}\PLH I$ equipped with the product topology. Define $\fG$ as the set of all {\IMC} currents with compact support in $\bar{Q}_\phi$, i.e. if $T\in \fG$, then its support is contained in $ \Omega \PLH [a,b]$ for some $[a,b]\subset (-\infty, A)$.  \\
 \indent The area minimizing problem in $Q_\phi$  is to find an {\IMC} $T_0\in \fG$ to realize the minimum of 
 \be \label{area-minimizing}
 \min \{ \bM(T): T\in \fG, \P T=\Gamma\}
 \ene 
 where $\psi(x)$ is a $C^1$ function on $\P\Omega$, $\Gamma=\{(x,\psi(x))\in \P\Omega\PLH I\}$ and $\bM(T)$  denotes the mass of $T$. See section 3 for related definitions.\\
 \indent Our main motivation is from the following three theorems. The first theorem due to Rado \cite{Rado32} and Tausch \cite{Tau80} states that if $\Gamma$ is a $C^2$ graph in the boundary of any convex Eucldiean cone then $\Gamma$ bounds a unique area minizing disk as a graph over $\R^n$. The second theorem due to Anderson \cite{And82} says that if $\Omega$ is a $C^2$ mean convex domain in the infinity boundary $S^{n}$ of Hyperbolic space $\H^{n+1}$, there is a local area minimizing minimal graph over $\Omega$ in $\H^{n+1}$ with infinity prescribed boundary $\P\Omega$. A third theorem is from Lin's thesis, section 4.1 in \cite{Lin85}. The author established the existence of area minimizing currents with compact support in a cylinder over bounded domains in $\R^{n}$ with a $C^1$ graphical boundary via bounded variation (BV) function theory.  Based on the above three results it is natural to ask how to solve the area minimizing problem \eqref{area-minimizing} in conformal cones.  \\
	\indent Similar kinds of generalized area minimizing problems in Euclidean cones to explore the existence of surfaces with prescribed mean curvature are considered in \cite{BH16},\cite{Fu03},\cite{LP03}, \cite{LP14}, \cite{Sau16},\cite{Sch04} and \cite{Sch05}, \cite{Bourni11} etc. For existence of area minimizing cones and some area minimizing problems, we refer to Lawlor \cite{Law91}, Morgan \cite{Mor02}, Zhang \cite{Zhang18}, Ding-Jost-Xin \cite{DJX16}, Ding \cite{Ding19} and refereneces therein. \\
\indent	A main difficulty to solve \eqref{area-minimizing} in conformal cones is how to describe minimal graphs in $Q_\phi$ with fixed boundaries for general $\phi(r)$. This is equivalent to solve the following Dirichlet problem to mean curvature equation: 
\be\label{mean:curvation} -div(\F{Du}{\omega})+n\F{\phi'(u(x))}{\phi(u(x))}\F{1}{\omega}=0 \quad \text{on}  \quad \Omega, 
\ene 
with $\omega=\sqrt{1+|Du|^2}$ and $u(x)=\psi(x)$ for $\psi(x)\in C(\P\Omega)$ (Corollary \ref{cor:min}). Here $div$ is the divergence of $\Omega$ and $Du$ is the gradient of $u$. A recent work of  Casteras-Heinonen-Holopainen \cite{CHH19} studied a similar form \eqref{mean:curvation} with a lower bound upon the Ricci curvature of $\Omega$ depending on $\phi(r)$.\\
\indent To overcome this difficulty we propose a topological condition for bounded domains different with \cite{CHH19} as follows. 
\begin{Def} \label{Def:NCM}Suppose $\Omega$ is a $n$-dimensional bounded Riemannian manifold with $C^2$ boundary. We say that $\Omega$ has the non-closed-minimal (NCM) property if it holds that \begin{enumerate}
		\item if $n\leq 7$, no closed embedded minimal hypersurface exists in $\bar{\Omega}$(the closure of $\Omega$);
		\item if $n >7$, no closed embedded minimal hypersurface with a closed singular set $S$ with $H^{k}(S)=0$ for any real number $k>n-7$ exists in $\bar{\Omega}$ where $H^k$ denotes the $k$-dimensional hausdorff measure on $\Omega$;
	\end{enumerate} 
\end{Def}
 All bounded $C^2$ domains in Euclidean spaces, Hyperbolic spaces and hemisphere (not itself) have the NCM property by the maximum principle from Ilmanen \cite{Ilm96} (remark \ref{remark:example}). \\
\indent The main result  from \cite{Zhou19} we will use in this paper is stated as follows. 
\bt [Theorem 1.7, \cite{Zhou19}]\label{tsh:thm} Suppose $\Omega$ is a $C^2$ mean convex domain with the NCM property. Then the Dirichlet problem of translating mean curvature equation
 \be\label{eq:teq}
\left\{\begin{split}
div(\frac{Du}{\omega}) &=\frac{\alpha}{\omega} \quad x\in \Omega \quad \omega =\sqrt{1+|Du|^2}\\
u(x)&=\psi(x)\quad x\in \partial\Omega
\end{split}\right . 
\ene 
is uniquely solved in $C^2(\Omega)\cap C(\bar{\Omega})$ for any $\alpha\in \R$ and any $\psi(x)\in C(\P\Omega)$. Here $div$ denotes the divergence of $\Omega$. 
 \et 
 Note that \eqref{eq:teq} in Euclidean spaces was firstly solved by White \cite{Whi15} and Wang \cite{Wang11}, see also Ma \cite{Ma18}. The above theorem yields the existence of minimal graphs with continuous boundaries in any mean convex conformal cone $Q_\phi$ with the NCM assumption if $\phi(r)$ is a positive $C^2$ function and satisfies 
  \be \tag{cA}
  \phi'(r)>0 \text{ on } (-\infty, A), |\F{\phi'(r)}{\phi(r)}|\leq \mu_0 \text{ on  } (-\infty,a)
  \ene
   for some $a<A$ and a positive constant $\mu_0$. By remark \ref{condition:Remark:one} the condition (cA) is sufficiently general. In addition if $\phi(r)$ satisfies 
   \be \tag{cB}
   (\log\phi(r))''\geq 0
   \ene 
   the solution is unique (see Theorem \ref{thm:A:B}). 
     With these assumptions in $Q_\phi$, we can construct a series of mean convex domains $D_{k,\alpha}$ as follows:
     \be
     \dk:=\{(x,t): x\in \Omega : t\in (u_{-k}(x), \alpha)\}
     \ene 
     where $u_{-k}(x)$ is the solution to \eqref{mean:curvation} with boundary data $\psi(x)-k$ (Lemma \ref{mean:convex:lm}). Consider the area minimizing problem restricted in $\bar{D}_{k,\alpha}$ (the closure of $\dk$) similar as that in \eqref{area-minimizing} in $\fG$. We show that the corresponding area minizing current is a boundary of a {\Ca} set restricted in $\bdk$(Lemma \ref{eq:lm:st}).  Moreover this current is disjoint with $\P D_{k,\alpha}\backslash \Gamma$(Lemma \ref{lm:disjoint}).  When $k$ is sufficiently large and $\alpha$ is close to $A$,  then the area minizing current is just the minimal graph in Theorem \ref{thm:A:B} (Theorem \ref{thm:MT:C}).  In the proof we need a regularity result of almost minimal sets in section 3 and a maximum principle of  $C^{1,\alpha}$ hypersurfaces in appendix A.  \\
     \indent Now we assume $\phi(r)$ satisfies the condition (cA) and the condition (cC) given by 
      \be \tag{cC}
      \lim_{r\rightarrow A-}\phi(r)=+\infty \quad \lim_{r\rightarrow A-}\inf_{s\in[r, A)}(\log \phi(r))'\geq c>0
      \ene   
    and $\Omega$ is mean convex with the NCM property. Note that  $\Omega\PLH\{A\}$ is the infinity boundary of $Q_\phi$.  As an application of Theorem \ref{thm:A:B} we can push minimal graphs with finite data into infinity to  obtain a complete minimal graph with prescribed infinity boundary data in $Q_\phi$ (see Theorem \ref{main:thm:D}). \\
 \indent Besides the conditions of $\phi(r)$, two main assumptions in section 2  and section 4 are the mean convex condition and the NCM assumption on $\Omega$. In the setting of translating conformal cones, the former condition can be relaxed as $\Omega$ is  contained in a larger $C^2$ domain $\Omega^*$ with mean convex and NCM assumptions. Then the area minimizing problem in \eqref{area-minimizing} is equivalent to a minimizing problem of some area functionals in \eqref{BV:minimum} (Theorem \ref{central:fact}). Then we show the existence of area minimizing current in \eqref{area-minimizing} in $\fG$ by the compactness of BV functions,  Theorem \ref{central:fact} and Theorem \ref{thm:MT:C}.\\
 \indent As for the NCM assumption,  it can not be removed if we hope the conclusions in Theorem \ref{thm:A:B} and Theorem \ref{thm:MT:C} hold. We show that in the case that $\Omega$ is the hemisphere $S^n_+$ and $\phi(r)=e^{\F{\alpha}{n}r}$ for $\alpha \geq n$, there is no bounded $C^2$  solution to  the Dirichlet problem in \eqref{eq:teq} for any $\psi(x)\in C(\P  S^n_+)$ (Theorem \ref{thm:one:two}) and no solution in $\fG$ to the problem \eqref{area-minimizing} (Theorem \ref{thm:two:two}). \\
  \indent This paper is organized as follows. In section 2 we discuss the Dirichlet problem of minimal surface equations in conformal cones. As an application we extend the result of Anderson in \cite{And82} on the existence of complete minimal graphs over the infinity boundary in Hyperbolic spaces into a class of conformal cones. \\
  \indent In  section 3  we collect preliminary facts on currents, BV functions, perimeter and (almost) minimal sets. We also discuss the regularity of almost minimal sets when their boundaries pass through the boundary of the intersection of two $C^2$ domains (Theorem \ref{regularity:key:thm}).  \\
 \indent In section 4 we solve the area minimizing problem in \eqref{area-minimizing} in the mean convex conformal cone with the NCM assumption. The results of Rado \cite{Rado32} and Tausch \cite{Tau80} are extended into the setting of conformal cones. \\
 \indent In section 5  we consider the problem  \eqref{area-minimizing} for translating conformal cones with non-mean convex boundaries. This extends Lin's result in Euclidean cylinders \cite{Lin85}. In section 6 we discuss examples to illustrate that the NCM assumption should not be removed for main results in this paper. In appendix A we record a maximum principle on $C^{1,\alpha}$ hypersurfaces via comparing their mean curvature. 
\section{Minimal surface equations from conformal cones}
 In this section we study the Dirichlet problem of minimal surface equations in conformal cones.  
 \subsection{Preliminiaries}
 Let $M$ be a $m$-dimensional complete Riemannian manifold with a metric $g$. Suppose $S$ is a $C^2$ hypersurface in $M$ with a normal vector $\vec{v}$. 
 \begin{Def} \label{def:mc} We call $div(\vec{v})$ as the \emph{mean curvature} of $S$ with respect to $\vec{v}$, written as $H_{S}$. Here $div$ is the divergence of $M$. If $H_S\equiv 0$, we say $S$ is minimal. 
 	\end{Def} 
 
   Let $\P \Omega$ be the boundary of a $C^2$ domain $\Omega$ in $M$. We always take its outward normal vector. With this convention  the mean curvature of the unit sphere $S^n$ in Euclidean spaces $\R^{n+1}$ is $n$. 
 \begin{Def} If $H_{\P \Omega}\geq 0$, we say that $\Omega$ is a mean convex domain. 
 	\end{Def}
 \indent  
 Let $f$ be a $C^2$ function on $M$. We write the manifold $M$ equipped with the metric $e^{2f}g$ as $\tilde{M}$.  The relation of the mean curvature of a hypersurface in two manifolds is given as follows. 
 \bl [Lemma 3.1 in \cite{ZhouA19}]\label{lm:mean:curvature} Suppose $S$ is a $C^2$ orientable hypersurface with the normal vector $\vec{v}$ in $M$. Let $\tilde{H}$ and $H$ denote the mean curvatuere of $S$ in $\tilde{M}$ and $M$ respectively.  Then 
 \be 
 \tilde{H}= e^{-f}(H+(m-1)df(\vec{v}))
 \ene  
 \el 
Now we recall the definition of conformal cones. 
\begin{Def}[Definition \ref{def:conformal:cone}] \label{def:cc:copy}Suppose $\Omega$ is an open bounded Riemanian manifold with $C^2$ boundary and a metric $\sigma$. Let $I$ be an open interval $(-\infty, A)$ where $A$ is a constant or $+\infty$. Let $\phi(r)$ be a $C^2$ positive function on $I$.  The set $  (\Omega\PLH I,  \phi^2(r)(\sigma+dr^2)) 
	$ is a conformal cone written as $Q_\phi$. \\
	\indent If $\phi(r)=e^{\F{\alpha}{n} r}$ for a constant $\alpha \geq 0$, we call such cone as a translating conformal cone. Let $C$ be an adjective. If $\Omega$ is $C$ (has the $C$ property), we call such cone as a $C$ conformal cone (with the $C$ property). 
\end{Def}
\br  If $\Omega$ is mean convex, we call $Q_\phi$ as a mean convex conformal cone. 
\er
\br
 Here  the term ``translating" is from the fact that a minimal graph in $Q_\phi$ for $\phi(r)=e^{\F{\alpha}{n}r}$ remains minimal under the translating motion $T_{t_0}(x,t)=(x,t+t_0)$ for a fixed $t_0$ and any $(x,t)\in Q_\phi$.
 \er
   With a parametrization  on $I$  our definition includes a large class of warped product manifolds (see Remark \ref{condition:Remark:one}). Let $S^n$ be the standard $n$-dimensional unit sphere with the metric $\sigma_n$. Then Euclidean spaces $\R^{n+1}$ and Hyperbolic spaces $\H^{n+1}$ can be written as follows. 
\begin{gather}
\R^{n+1}\backslash\{0\}:=(S^n\PLH\R, e^{2r}(\sigma_n+dr^2))\\
\H^{n+1}\backslash\{0\}:=(S^n\PLH (-\infty, 0), \F{4e^{2r}}{(1-e^{2r})^2}(\sigma_n+dr^2))\label{model:hyperbolic}
\end{gather}
Thus all cones in Ecludiean space are translating conformal cones. 
\indent 
 Now we apply Lemma \ref{lm:mean:curvature} into the case of conformal cones. 
\bl \label{lm:mc} Let $Q_\phi$ be a conformal cone given in Definition \ref{def:cc:copy}. Suppose $u(x)\in C^2(\Omega)$ and $\Sigma=(x,u(x))$. Then the mean curvautre of $\Sigma$ {\wrt} the upward normal vector in $Q_\phi$ is written as 
\be 
     H_\Sigma=\F{1}{\phi(u(x))}\{-div(\F{Du}{\omega})+n\F{\phi'(u(x))}{\phi(u(x))}\F{1}{\omega}\}
\ene 
where $\omega=\sqrt{1+|Du|^2}$ and $div$ is the divergence of $\Omega$. 
\el 
\bp  In the product manifold $\Omega\PLH \R $ the upward normal vector of $\Sigma$ is $\vec{v}=\F{\P_r- Du}{\omega}$. Here $Du$ is the gradient of $u$ on $\Omega$. A direct verification shows that the mean curvature of $\Sigma$ in the product manifold is $-div(\F{Du}{\omega})$ (see \cite{Zhou17a}).  Our conclusion follows from Lemma \ref{lm:mean:curvature} and $e^{f}=\phi(r)$. \ep 
As a corollary we have 
\begin{cor}\label{cor:min} Suppose $u\in C^2(\Omega)$.  Then its graph is minimal in $Q_\phi$ if and only if $u(x)$ satisfies that 
	\be \label{eq:min}
	 -div(\F{Du}{\omega})+n\F{\phi'(u(x))}{\phi(u(x))}\F{1}{\omega}=0
	\ene 
	where $\omega=\sqrt{1+|Du|^2}$ and $div$ is the divergence of $\Omega$. 
	\end{cor}
 \br 
 In the case of $\phi(r)=e^{\F{\alpha}{n}r}$,  \eqref{eq:min} is called as the \emph{translating mean curvature equation}. 
 \er 

\subsection{Minimal surface equations} Recall that the NCM property is defined as follows.
\begin{Def} [Definition \ref{Def:NCM}] Suppose $\Omega$ is a $n$-dimensional Riemannian manifold with $C^2$ boundary. We say $\Omega$ has the non-closed-minimal (NCM) property if it holds that \begin{enumerate}
		\item if $n\leq 7$ , no closed embedded minimal hypersurface exists in $\bar{\Omega}$;
		\item if $n > 7$ , no closed embedded minimal hypersurface with a closed singular set $S$ with $H^{k}(S)=0$ for any real number $k>n-7$ exists in $\bar{\Omega}$ where $H^k$ denotes the $k$-dimensional Hausdorff measure on $N$;
	\end{enumerate}
\end{Def}
\br \label{remark:example}All bounded $C^2$ domains in Euclidean spaces,  Hyperbolic spaces and the hemisphere $S^n_+$ (not itself ) have the NCM property.  By White \cite{Whi09}, minimal surfaces in those domains have similar isoperimetric inequalities to those in Euclidean spaces. It is also similar to the condition of Giusti \cite{Giu78} on the existence of prescribed mean curvature graphs. 
\er 
 Now we consider the following Dirichlet problem 
 \begin{equation}\label{main:equation:A}
 \left\{\begin{split}
 div(\frac{Du}{\omega}) &=\frac{n\phi^{'}(u(x))}{\phi(u(x))\omega} \quad x\in \Omega \quad \omega =\sqrt{1+|Du|^2}\\
 u(x)&=\psi(x)\quad x\in \partial\Omega
 \end{split}\right . 
 \end{equation}
 where $\psi(x)$ is a continous function on $\P\Omega$ and $div$ is the divergence of $\Omega$. The main result in this section is stated as follows. \bt\label{thm:A:B}  Let ${\Omega}$ be a $C^2$ bounded domain with mean convex  boundary and the NCM property.  Suppose a positive $C^2$ function $\phi(r)$ satisfies 
 \be \tag{cA}
 \phi'(r)>0 \text{ on } (-\infty, A), |\F{\phi'(r)}{\phi(r)}|\leq \mu_0 \text{ on  } (-\infty,a)
 \ene
 for some $a<A$ and a positive constant $\mu_0$. Then the Dirichlet problem $\eqref{main:equation:A}$ admits a solution $u\in C^2(\Omega)\cap C(\bar{\Omega})$ for any $\psi(x)\in C(\P\Omega)$ satisfying $\psi(x)<A$ on $\P\Omega$.   \\
 \indent In addition if 
  \be \tag{cB}
 (\log\phi)^{''}(r)\geq 0\text{ for all } r \in (-\infty,A)
 \ene  such solution is unique. 
 \et
 \br We believe that the condition (cB) is just a sufficient condition to obtain the unique
 ness result but not a necessary condition. 
 \er
  \br\label{condition:Remark:one}The condition (cA) is very general in terms of warped product manifolds. Let $m(s)$ be a positive $C^2$ function on an interval $I_1:=(a,b)$ with $m(a)=0$. Then  a cone in a warped product manifold is defined as 
 \be \label{dst}
 \{ \Omega\PLH I_1,   ds^2+m^2(s)\sigma\}
 \ene 
 where $\sigma$ is a metric on $\Omega$.  Now let $r$ be a new parameter on $I$ given by $r=\int \F{1}{m(s)}ds$. Suppose $s=s(r)$ is the corresponding inverse function. Now  $\phi(r)=m(s(r))$ and \eqref{dst} is rewritten as 
 \be 
 \{ \Omega\PLH (-\infty, A),   \phi^2(r)(\sigma+dr^2)\}
 \ene 
 for some constant $A$. As a result $m'(s)=\F{\phi'(r)}{\phi(r)}$. The condition (cA) for $\phi(r)$  is equivalent to $m'(s)>0$ on $I_1$ and $m'(s)$ are uniformly bounded on some subinterval $(a,b_0)$ for $b_0<b$.\\
 \indent  The condition (cB) for $\phi(r)$ is equivalent to $m''(s)\geq 0$ for all $s\in (a,b_0)$. 
 \er

% \br
% To prove Theorem \ref{thm:A:B} , we'll first suppose that $\Omega$ is a $C^3$ domain and $\psi\in C^3({\bar{\Omega}})$
% \er
\bp  Under the condition (cB) the uniqueness of the solution to \eqref{main:equation:A} is obvious from the maximum principle of elliptic equations. So we skip its proof. \\
\indent Now we take the condition (cA) and focus on the existence of the solution to \eqref{main:equation:A}. Its proof is a standard process according to section 11.3 in \cite{Tru67}). In the following, we use $C(a,b,c,\cdots,)$ to denote a constant only depending on $a,b,c,\cdots$. \\
 \indent First we assume that $\psi(x)\in C^3(\bar{\Omega})$ and $H_{\P\Omega}>0$ on $\P\Omega$.  Let $u(x)$ be a function in  $C^2(\Omega)\cap C(\bar{\Omega})$ solving the Dirichlet problem \eqref{main:equation:A}. By the condition (cA) and the maximum principle we have for all $x\in \Omega$, 
 $$
 u(x)\leq B:=\max_{\P\Omega}\psi(x) 
 $$
 Again by the condition (cA) there is a $\beta>0$ such that 
 $n\max_{r\in (-\infty,B]}|\F{\phi'(r)}{\phi(r)}|\leq \beta$. According to Theorem \ref{tsh:thm}, there is a $w(x)\in C^2(\Omega)$ solving 
 \be \label{eq:n:B}
 \left\{\begin{split}
 	div(\frac{Du}{\omega}) &=\frac{\beta}{\omega} \quad x\in \Omega \quad \omega =\sqrt{1+|Du|^2}\\
 	u(x)&=\min_{\P\Omega}\psi(x)\quad x\in \partial\Omega
 \end{split}\right .
 \ene  
Consider $v(x)=u(x)-w(x)$. Suppose $v(x)$ achieves $\min_{\bar{\Omega}}v(x)$ at $x_0\in \Omega$.  Then from \eqref{main:equation:A} $v(x)$ satisfies the following type nonlinear equation 
\be\label{eq:mi:step}
   a^{ij}(Du) v_{ij}+b^j v_j=\F{n\phi^{'}(u(x))}{\phi(u(x))\sqrt{1+|Du|^2}}-\F{\beta}{\sqrt{1+|Dw|^2}}
   \ene 
  Here $\{a^{ij}(Du)\}$ is a positive definite matrix near $x_0$ and $v_{ij}$ denotes the covariant derivatives of $v$. Note that $Du=Dw$ at $x_0$.  From the maximum principle, the definition of $\beta$ and \eqref{eq:mi:step} at $x_0$ we have 
  \be 
  0\geq a^{ij}(Du)v_{ij}\geq 0
  \ene 
  This is a contradiction to the weak maximum principle. Thus $v(x)$ takes its minimum at $\P\Omega$. Because $u(x)=\psi(x)\geq w(x)$ on $\P\Omega$, then $v(x)\geq 0$ on $\bar{\Omega}$. Namely 
  \be \label{cobound}
      \max_{\P\Omega} \psi(x)\geq u(x)\geq w(x)
  \ene 
  where $w(x)$ is only depending on $\min_{\P\Omega}\psi(x)$ and $\beta$. \\
  \indent  Since $\Omega$ is strictly mean convex, i.e. $H_{\P\Omega} >0$,  arguing exactly Theorem 14.6 in \cite{GT01} yields that 
  \be\label{eq:boundary:estimate}
  \max_{\P \Omega}|Du|\leq C(\Omega, |\psi(x)|_{C^2(\bar{\Omega})})
 \ene 
  Note that $u(x)$ satisfies 
  \be 
    \tilde{\sigma}^{ij}u_{ij}=n\F{\phi'(u)}{\phi(u)}
  \ene 
  where $\tilde{\sigma}^{ij}=\sigma^{ij}-\F{u^iu^j}{1+|Du|^2}$ and $u^{i}=\sigma^{ik}u_i$. Here $\sigma^{ij}$ is the inverse matrix of the metric $\sigma=\sigma_{ij}dx^idx^j$ on $\Omega$ with respect to a local coordinate $\{x^1, \cdots, x^n \}$. 

 Recall that $\omega=\sqrt{1+|Du|^2}$. By Lemma 3.5 in \cite{zhou18b}, we just view $u(x)$ as a $C^2$ function indepent of time $t$ and obtain that 
 \be\label{eq:borrow:lemma}
 \begin{split}
 	\tilde{\sigma}^{ij}\omega_{ij}&-\frac{2}{\omega}\tilde{\sigma}^{ij}\omega_i\omega_j\\
 	&=(|A|^2+Ric(\frac{Du}{\omega}, \frac{Du}{\omega}))\omega 
 	+n \left<\frac{Du}{\omega}, D(H\omega)\right>
 	\end{split}
 	\ene 
 	where $
 	|A|^2=\frac{1}{\omega^2}\tilde{\sigma}^{ik}\tilde{\sigma}^{jl}u_{ij}u_{kl}$, $  \tilde{\sigma}^{ij}={\sigma}^{ij}-\frac{u^i u^j}{1+|Du|^2}
 	$
 	 and Ric is the Ricci curvature of $\Omega$ and $H=div(\F{Du}{\omega})$ by \eqref{main:equation:A}. \\
 	\indent 
 Let $\eta$ be the function $e^{Ku}$ where $K$ is a sufficiently large constant determined later. Our purpose is to show that $\eta\omega$ is uniformly bounded when $K$ is a large constant only depending on $\max_{\bar{\Omega}}|u|$ and the metric on $\Omega$.\\
 \indent  Suppose that $\eta\omega$ attains its maximum in $\bar{\Omega}$ at $y_0\in \Omega$. At $y_0$ we have  $\omega_i\eta+\omega\eta_i=0$ for any $i\in \{1,\cdots,n\}$. Furthermore assume that $|Du|\ge 1$ at $y_0$.  Otherwise nothing needs to prove.  A direct computation (see section 3 in \cite{zhou18b}) yields that 
 	\begin{align*}
 	&\ts^{ij}(\omega\eta)_{ij}
 	=\omega \ts^{ij}\eta_{ij} +\eta \ts^{ij} \omega_{ij} -\frac{2}{\omega}\tilde{\sigma}^{ij}\omega_i\omega_j \eta \\
 	&=\eta\omega\left\{|A|^2+Ric(\frac{Du}{\omega}, \frac{Du}{\omega})+n{(\frac{\phi{'}}{\phi})}'\frac{|Du|^2}{1+|Du|^2}+ \ts^{ij}(K^2 u_i u_j + K u_{ij})\right\}\\
 	&=\eta\omega\left\{|A|^2+Ric(\frac{Du}{\omega}, \frac{Du}{\omega})+nK\frac{\phi{'}}{\phi}+(K^2+n(\frac{\phi{'}}{\phi})‘)’\frac{|Du|^2}{1+|Du|^2}\right\}\\
 	&\geq\eta\omega(\frac{1}{2}K^2-C) 
 	\end{align*}
 	where $C$ is a positive constant only depending on $\max_{\bar{\Omega}}u$,  $\min_{\bar{\Omega}}u$ and the lower bound of Ricci curvature on $\bar{\Omega}$. Taking $K$ sufficiently large, we have $\ts^{ij}(\omega\eta)_{ij}>0$ at $y_0$. This contradicts that $\eta\omega$ attains its maximum on $\bar{\Omega}$ at $y_0$. Combining this with \eqref{eq:boundary:estimate},  we have 
 	\be \label{eq:estimate}
 	\max_{\bar{\Omega}}|Du|\leq C(\max_{\bar{\Omega}}u,\min_{\bar{\Omega}}u, |\psi(x)|_{C^2(\bar{\Omega})})
 	\ene 
  Let $u^s(x)$ be the solution to the following equation 
 	\be 
 	\left\{\begin{split}
 	div(\frac{Du}{\omega}) &=s\frac{n\phi^{'}(u(x))}{\phi(u(x))\omega} \quad x\in \Omega \quad \omega =\sqrt{1+|Du|^2}\\
 	u(x)&=s\psi(x)\quad x\in \partial\Omega
 	\end{split}\right . 
 	\ene 
 	for any $s\in  [0,1]$ in $C^2(\Omega)\cap C(\bar{\Omega})$. Then arguing as \eqref{cobound},
 $$	 \max_{\P\Omega} \psi(x)\geq u^s(x)\geq w(x)$$
 for all $s\in [0,1]$. Therefore a similar derivation as in \eqref{eq:estimate} yields the following estimate 
 \be 
 	\max_{\bar{\Omega}}|Du^s|\leq C(\max_{\bar{\Omega}}|w(x)|, |\psi(x)|_{C^2(\bar{\Omega})})
 \ene 
 which is independent of $s$. 
 		By the standard Schauder estimates, $\sup_{\bar{\Omega}}|D^2 u^s|$ is also uniformly bounded for any $s\in [0,1]$.  
 	By the continuous method in Theorem 11.3 of \cite{GT01} (see section 11.3 in \cite{GT01}) the Dirichlet problem \eqref{main:equation:A} is solved for $\psi(x)\in C^3(\bar{\Omega})$ and $H_{\P \Omega}>0$ on $\P \Omega$. \\
 	\indent Now suppose $\psi(x)\in C^3(\bar{\Omega})$ and $H_{\P\Omega}\geq 0$. We evolve $\P\Omega$ with the mean curvature flow,  $\Sigma_t$ exists smoothly on $t\in [0,T]$ for some $T>0$. By corollary 3.5 (i) in \cite{Hui86},  the mean curvature $H$ along $\Sigma_t$ satisfies 
 	\be
 	\P_t H=\Delta H+H(|A|^2+Ric(\vec{v},\vec{v})) 
 	\ene 
 	where $|A|^2$ is the norm of second fundamental form of $\Sigma_t$, $Ric$ is the Ricci curvature of $\Omega$ and $\vec{v}$ is the normal vector of $\Sigma_t$. Since $H_{\P\Omega}\geq 0$ at time $t=0$, the maximum principle implies that $H_{\Sigma_t}>0$ for all $t\in (0,T)$. Moreover the domains enclosed by $\Sigma_t$, $\Omega_t$, are contained in $\Omega$. Thus $\{\Omega_t\}_{t\in (0,T)}$ have mean convex boundaries and the NCM properties.  \\
 	\indent Since $\Sigma_t=\P\Omega_t$ converges to $\P\Omega$ in the $C^2$ sense,  we can construct a series of  $\{\psi_t(x)\in C^3(\bar{\Omega}_t)\}$ for $t\in (0,T)$ and converge to $\psi(x)$ as $t\rightarrow 0$ on $\bar{\Omega}$. Thus we can assume $\psi_t(x)$ is uniformly bounded. By the perivious argument there is a family of $\{u_t(x)\}$ satisfies  $div(\F{Du}{\omega})=\F{\phi'(u(x))}{\phi(u(x))}\F{1}{\omega}$ on $\Omega_t$ with $u_t(x)=\psi_t(x)$ on $\Omega_t$. Arguing exactly as \eqref{eq:n:B}-\eqref{cobound}, then 
 	  \be
 	    \max_{\bar{\Omega}_t}|u_t(x)|\leq  C(\max_{\bar{\Omega}}|\psi(x)|,\beta) 
 	 \ene
 	  for all $t\in (0,T]$ and $x\in \Omega_t$.\\
 	  \indent  By Lemma \ref{lm:interior:estimate}, then for any fixed $x\in \Omega$, $|Du_t (x)|$ is locally uniform bounded near $x$. By the Schauder estimate, so is the $C^2$ norm of $u_t(x)$ near $x$. Thus $u_t(x)$ converges to a function $u(x)$ in the locally  $C^2$ sense on $\Omega$ as $t\rightarrow 0$. Thus $u(x)=\psi(x)$ on $\P\Omega$ and satisfies that $div(\F{Du}{\omega})=\F{\phi'(u(x))}{\phi(u(x))}\F{1}{\omega}$. \\
 	  \indent For any $\psi(x)\in C(\P\Omega)$, we can construct a sequence $\{\psi_j(x)\}\in C^3(\bar{\Omega})$ such that this sequence converges to $\psi(x)$ in $C(\P\Omega)$. A similar approximating process above yields the Dirichlet problem \eqref{main:equation:A} with boundary data $\psi(x)$. The proof is complete. \ep

The following interior estimate of mean curvature equations is based on a work of Wang in \cite{Wang98} (see also  Lemma 2.3 in \cite{CHH19}).
 \bl [Theorem 1.4 in Gui-Jian-Ju in \cite{GJJ10}] \label{lm:interior:estimate}Let $B_r(x_0)$ be an embedded ball in $\Omega$ and $u(x)\in C^2(B_r(x_0))$ satisfies that 
 \be \label{eq:condition:assumption}
 div(\F{Du}{\omega})=\F{f(u(x))}{\omega}
 \ene 
 where $\omega=\sqrt{1+|Du|^2}$ and $f(r)$ is a $C^1$ function on $\R$. Let $a<b$ be two finite constants such that $a\leq u(x)\leq b$ on $B_r(x_0)$. Then 
 \be 
 |Du(x_0)|\leq C(r, |f|_{C^1[a,b]}, \max_{\bar{\Omega}}|Ric|)
 \ene 
 where $|f|_{C^1[a,b]}$ is the $C^1$ norm of $f$ on $[a,b]$. 
 \el 
 \br  The $f(r)$ in the right side of \eqref{eq:condition:assumption} does not affect the derivation of Theorem 10 in \cite{GJJ10}.  In their case they just consider $f\equiv 1$. 
 \er 
Combining the existence and the uniqueness in Theorem \ref{thm:A:B} we obtain the following continous result with respect to the boundary data. 
 \bl \label{lm:continous} Take the same assumptions as in Theorem \ref{thm:A:B}.  Suppose $\psi(t, x)$ is a continous function on $\P\Omega \PLH[0,1)$. For $t\in [0,1)$ let $u_t(x)$ be the solution to the Dirichlet problem
 \begin{equation}\label{st:eq:w}
 \left\{\begin{split}
 div(\frac{Du}{\omega}) &=\frac{n\phi{'}(u)}{\phi(u)\omega} \quad x\in \Omega \quad \omega =\sqrt{1+|Du|^2}\\
 u(x)&=\psi(t,x)\quad x\in \partial\Omega
 \end{split}\right.  
 \end{equation}
 Then $\{u_{t}(x)\}_{t>0}$ converges to $u_{0}(x)$ in the sense of $C(\bar{\Omega})$ and locally in $C^2(\Omega)$ as $t\rightarrow 0$.  
 \el 
 \subsection{The infinity Plateau problem} Now we immediately give an application of Theorem \ref{thm:A:B}.  Consider a conformal cone given by 
 \be \label{condition:A}
 Q_\phi :=\{\Omega\PLH(-\infty, A),  \phi^2(r)(\sigma+dr^2)\}
 \ene 
 where $A$ is a finite number and $\phi(r)$ is a $C^2$ positive function satisfying  that 
 \begin{gather}
 \tag{cA}
 \phi'(r)>0 \text{ on } (-\infty, A), |\F{\phi'(r)}{\phi(r)}|\leq \mu_0 \text{ on  } (-\infty,a)\\
 \tag{cC}
 \lim_{r\rightarrow A-}\phi(r)=+\infty \quad \lim_{r\rightarrow A-}\inf_{s\in[r, A)}(\log \phi(s))'\geq c>0
 \end{gather}
 Here $\mu_0, a$ are two constants with $\mu_0>0$ and $a<A$. 
 An example of conformal cones satisfying (cA) and (cC) is the cones in Hyperbolic spaces (see \eqref{model:hyperbolic}).  \\
 \indent In the above setting,  $\Omega\PLH\{A\}$ is referred as the inifnity boundary of $Q_\phi$.  For any $n$-rectifible set $\Gamma\subset \Omega\PLH\{A\}$, the infinity Plateau problem is to find a complete minimal hypersurface in $Q_\phi$  asymptotic to $\Gamma$. For an example in a hyperbolic space more details see \cite{And82}. The second main result in this section is given as follows.
 \bt  \label{main:thm:D} Let $\Omega$ be a $C^2$ mean convex bounded domain with the NCM property. Suppose $Q_\phi$ is a conformal cone satisfying conditions (cA),  (cC). Then there is a smooth function $u(x)$ over $\Omega$ such that its graph $\Sigma$ is a minimal graph in $Q_\phi$ with the infinity boundary $\P\Omega\PLH \{A\}$. 
 \et
 \br There is little geometric information of $Q_\phi$ in this general setting comparing to Hyperbolic spaces.  This result can be viewed as a generalization of Theorem 10 in \cite{And82}.  See also Theorem 2.1 in \cite{LFH89}. Its uniqueness will be considered in the future \cite{CSZ20}. In the proof we just apply the existence in Theorem \ref{thm:A:B}. Thus we do not need the condition (cB) here. 
 \er 
 \bp Let $t\in (-\infty , A)$.  By Theorem \ref{thm:A:B}, there is a $C^2$ function $u_t(x)\in C^2(\Omega)\cap C(\bar{\Omega})$ satisfying 
 \be \label{eq:mid:rstu}
 \left\{\begin{split}
 	div(\frac{Du}{\omega}) &=\frac{n\phi^{'}(u)}{\phi(u)\omega} \quad x\in \Omega \quad \omega =\sqrt{1+|Du|^2}\\
 	u(x)&=t\quad x\in \partial\Omega
 \end{split}\right . 
 \ene 
 \indent Denote the graph of $u_t(x)$ by $\Sigma_t$. By the condition (cA),  arguing as \eqref{cobound}  there is a global constant $\mu_1, t_1<A$ such that $u_t(x)\geq \mu_1$ for any $t\in [t_1, A)$. \\
 \indent  We claim that $\{u_t(x)\}_{t\in [t_1,A)}$ has a local upper bound strictly less than $A$ in $\Omega$. \\
 \indent We use $B_r(x)$ to denote the open ball centerred at $x$ with radius $r$. Now fix $x_0\in \Omega$. There is a $r_0>0$ depending the geometry of $\Omega$ near $x_0$ and $c$ such that $B_{r_0}(x_0)\subset \Omega$ is mean convex with the NCM property. By the condition (cC) choose $\beta$ sufficiently small such that  
 \be
 \beta<\inf_{s\in [t_1,A)}(\log\phi(s))'
 \ene  By Theorem \ref{thm:A:B} the Dirichlet problem 
 \be \label{eq:mid:st}
 \left\{\begin{split}
 	div(\frac{Du}{\omega}) &=\F{\beta}{\omega}\quad x\in  B_{r_0}(x_0)\\
 	u(x)&=A\quad x\in \P B_{r_0}(x_0)
 \end{split}\right . 
 \ene 
 has a unique solution $u_\beta(x)$ in $C^2(\Omega)\cap C(\bar{\Omega})$. Furthermore we can take $\beta$ enough small such that $u_{\beta}(x)\geq t_1$ for all $x\in B_{r_0}(x_0)$. \\
 \indent For any $t\in [t_1, A)$, define $v(x)=u_{\beta}(x)-u_{t}(x)$ in $B_{r_0}(x_0)$. Then $v(x)$ satisfies that 
 \be \label{eq:st:md}
 a^{ij}(Du) v_{ij}+b^j v_j=\F{\beta}{\sqrt{1+|Du_\beta|^2}}-\frac{n\phi^{'}(u_t(x))}{\phi(u_t(x))\sqrt{1+|Du_t|^2}}
 \ene  
 where $a^{ij}(Du)$ is a positive definite matrix. Suppose $v(x)$ achieves its minimum on $\bar{B}_{r_0}(x_0)$ at $y_0\in B_{r_0}(x_0)$. If $u_t(y_0)<t_1$ we have $v(y_0)=u_{\beta}(y_0)-u_t(y_0)>0$. We conclude $v(x)>0$ on $B_{r_0}(x_0)$. If $u_t(y_0)\geq t_1$, by \eqref{eq:st:md} one sees that at $y_0$
 \be 
 0\leq a^{ij}(Du)v_{ij}=\F{1}{\sqrt{1+|Du_t|^2}}(\beta-\frac{n\phi^{'}(u_t(y_0))}{\phi(u_t(y_0))}\leq 0
 \ene 
 This is a contradiction to the weak maximum principle. Thus in both cases we have $v(x)\geq 0$ on the ball $B_{r_0}(x_0)$. Thus $\mu_1\leq u_t(x)\leq u_{\beta}(x)<A$ on the closure of $B_{\F{r_0}{2}}(x_0)$ for any $t\in  [t_0, A)$. \\
 \indent Then there is a sequence $\{t_j\}_{j=1}^\infty$ such that $t_j\rightarrow A-$ and $u_{t_j}(x)$ converges uniformly to $u_{A}(x)$ in $B_{\F{r_0}{2}}(x_0)$ in the $C^0$ norm. By Lemma \ref{lm:interior:estimate}  and the standard Schauder estimate, such convergence holds in the $C^2$ norm on $B_{\F{r_0}{2}}(x_0)$. Because  $x_0\in \Omega $ is chosen arbitrarily,  we can choose a countable open ball in $\Omega$ such that their union is $\Omega$ and the above convergence holds in each ball.   This means there is a sequence $\{t_j\}_{j=1}^\infty$ such that $t_j\rightarrow A-$ and $u_{t_j}$ converges locally to $u_A(x)$ in the $C^2$ norm satisfying \eqref{eq:mid:rstu}. Let $\Sigma$ be the graph of $u_A(x)$. Thus $\Sigma$ is minimal in $Q_\phi$. From the definition of $u_t(x)$ the boundary of $\Sigma$ is $\P\Omega \PLH\{A\}$. \\
 \indent The proof is complete.  
 \ep 
 \section{Currents and almost minimal boundary}
 In this section we collect prelinimary facts on currents and almost minimal boundaries. The main results in this section are Theorem \ref{thm:constant:current} and Theorem \ref{regularity:key:thm}. They play an essential role in the proof of Theorem \ref{thm:MT:C}. 
 \subsection{Currents} 
 Our main references are the book of Simon \cite{Sim83} and Lin-Yang \cite{LY02}.  Let $U$ be an open domain in a Riemannian manifold $M$ and $H^j$ denote the $j$-dimensional Hausdorff measure. Suppose $k$ is an integer. Let $D^k(U)$ be  the set of all $k$-smooth form with compact support in $U$. 
 \begin{Def}
 	A k-current $T$ is a linear continuous functional on $D^k(U)$.   The mass of the current $T$ in $U$ is 
 	\be
 	\bM_U(T):=sup\{T(\omega): \la \omega, \omega\ra \leq 1, \omega \in D^k(U) \}
 	\ene 
 	where $<,>$ denotes the usual pairing of $k$-form. 
 	\end{Def}
  If there is no ambiguity about $U$, we write $\bM (T)$ instead of $\bM_U(T)$. By the Radon-Nikodym Theorem there is a Radon measure $\mu_T$ on $M$ such that for any $\omega\in D^k(U)$, 
 \be 
 T(\omega)=\int_{M}\la \omega, \vec{T}\ra d\mu_T
 \ene
 where $\vec{T}$ is a unit $k$-form a.e. $\mu_T$. Thus $\bM_U(T)=\int_U d\mu_T$. Then we can also discuss the mass of a k-currents in Borel sets. 
 \begin{Def} For a $k$-current $T$, its boundary $\P T$ in $U$ is  a current acting on  $D^{k-1}(U)$ such that $\P T(\omega)=T( d\omega)$ for any $\omega\in C^{k-1}(U)$ where $d$ is the differential operator on smooth forms.\\
 	\indent The \emph{support}, spt(T), of $T$ is the relatively closed subset of $M$ defined by 
 	$$
 	spt T=M\backslash (\cup V)
 	$$
 	where the union is over all open set $V\subset \subset M$ such that $T(\omega)=0$ for any smooth form $\omega$ with $spt \omega\subset V$. 
 	\end{Def} 
  The concept of $k$-current is a generalization of the $k$-dimensional oriented submanifold $S$ in $M$. Suppose $\eta(x)$ is a local orientation of $S$. Then there is a corresponding $k$-current $[[S]]$ is defined by 
 \be \label{def:S}
 [[S]](\omega)=\int_{S}\la \omega(x),\eta(x)\ra d H^k \ene 
 for any smooth vector $\omega$ with compact support. 
 \br\label{rm:def:S} We can also define $[[S]]$ when $S$ is a Borel set approximated by a series of open sets in $M$ if choose $\eta$ as the orientation of $M$.  \er 
 The following concept is very useful in geometric measure theory. 
  \begin{Def} A set $E\subset M$ is said to be countable $k$-rectifiable if 
 	$$ E\subset E_0\cup_{j=1}^\infty F_j(E_j)$$
 	where $H^k(E_0)=0$ and $F_j:E_j\subset \R^k\rightarrow M$ is a Lipschitz map for each $j$. 
 \end{Def}
 Now we can define an integer multiplicity rectifiable $n$-current. 
 \begin{Def} Let $T$ be a $k$ current in $M$, we say that $T$ is an integer multiplicity rectifiable $n$-recurrent (integer multiplicity current) if 
 	\be 
 	T(\omega)=\int_S\la \omega,\eta\ra \theta(x)d H^k(x) 	\ene 
 	where $S$ is a countable $k$-rectifiable subset of $M$, $\theta$ is a positve locally $H^k$-integrable function which is integer-valued, and $\eta$ is a $k$ form $\tau_1\wedge\cdots\wedge \tau_k$ oriented the tangent space of $S$ a.e. $H^k$. $T$ is also written as
 	$\tau(S,\theta, \eta)$. 
 	\end{Def}
 \br \label{rk:mark:notation} According to Remark \ref{def:S}, for any open $k$-submanifold $M'$, $[[M']$ is an {\IMC} just choosing $\eta$ as the orientation, is equal to $\tau(M', 1,\eta)$\\
 \indent If the dimension of an {\IMC} $T$, $\tau(S, \theta,\eta)$, has the same dimension as that of $M$. We always choose $\eta$ as the volume form of $M$. In this case $T$ is written as $\tau(S,\theta)$. 
 \er 
 A good property of {\IMC}s is their compactnesss theorem firstly obtained by Federer and Fleming \cite{Sim83}.
 \bt [Feder-Fleming Compactness Theorem]\label{compact:thm} Suppose $\{T_j\}_{j=1}^\infty$ is a sequence of {\IMC}s with $$\sup\{\bM_W(T_j)+\bM_W(\P T_j)\}<\infty$$ for any $W\subset\subset M$, then there is an integer mulitiplicity current $T$ such that $T_j$ converges weakly to $T$ and $\bM_W(T)\leq \lim_{j\rightarrow +\infty} \sup_{i\geq j} \bM_W(T_i)$. \et 
 A useful way to construct {\IMC}s is the pushforward of local Lipschitz maps. 
 \begin{Def} Let $U, V$ be two open sets in (different) Riemannian manifolds. Suppose $f:U\rightarrow V$ is local Lipschitz,  $T=\tau(S, \eta, \theta)$ is an $k$ {\IMC} and $f|spt T$ is proper,  then we can define $f_\#T$ by
 	\be 
 	f_{\#}T(\omega)=\int_S\la \omega|_{f(x)}, df_{\#}\eta\ra \theta(x) dH^k(x)
 	\ene 
 	\end{Def}
Now we proceed the derivation in section (26.26, \cite{Sim83}) under the setting of a conformal cone $Q_\phi$ in Definition \ref{def:conformal:cone}. Let $\{(x,t):x\in \Omega, t\in I=(-\infty, A)\}$ be a coordinate in $Q_\phi$.  Now for any $t\leq 0$ we define $h:(-\infty, 0)\PLH  Q_\phi\rightarrow Q_\phi$ as $h(t, (x,r))=(x, r+t)$.  Note that $h$ is proper and local Lipschitz {\wrt} $Q_\phi$. Suppose $T$ is a $k$ {\IMC} with compact support in $Q_\phi$. Then $h_{\#}([[(-\infty, 0)]]\PLH T)$ is well-defined. Because  $h(0,(x,r))=(x,r)$, we have 
 \begin{align*}
 \P h_{\#}([[(-\infty, 0)]]\PLH T)&=h_{\#}(\P ([[(-\infty, 0)]]\PLH T))\\
 &= h_{\#}(\{0\}\PLH T)-h_{\#}((-\infty,0)\PLH \P T)\\
 &=T-h_{\#}((-\infty,0 )\PLH \P T)
 \end{align*}
The case $\P T=0$ yields the following result. 
 \bt \label{thm:constant:current} Now let $\Omega,\phi, Q_\phi$ be defined in Definition \ref{def:conformal:cone}.   Suppose $n$ is the dimension of $\Omega$ and $k\leq n$ is an positive integer. Let $T$ be a $k$ {\IMC} in the conforml cone $Q_\phi$ with compact support satisfying $\P  T=0$.  Then there is a $k+1$ {\IMC} $R$ in $Q_\phi$ such that $\P R =T$. Here  $spt (R)$ may be noncompact in $\bar{\Omega}\PLH(-\infty, A)$.
 \et 
 \subsection{Perimeter and Regularity}
 In this subsection we recall some preliminar facts on BV functions, perimeter and the regularity of almost minimal boundary for later use. The main references are \cite{Giu84}, \cite{LY02} ,\cite{Sim83} and \cite{Zhou19}. \\
 \indent Let $M$ be a {\RM} with a metric $g$ and dimension $n+1$. Suppose $W\subset M$ is an open set. We denote the set of vector fields (continuous functions) on $N$ with compact support in $W$ by $T_0W$ ($C_0(W))$.  
 \begin{Def} 
 	For any 
 	$u\in L^1(W)$, the variation of $u$ is defined as 
 	\be\label{def:variation}
 	||Du||_M(W)=\sup\{\int_{W}udiv(X)dvol: X\in T_0(W), \la X,X\ra\leq 1\}
 	\ene 
 	where $div$ and $dvol$ are the divergence and the volume of $M$ respectively. 
 	We say $u\in BV(W)$ if $||Du||_{M}(W)$ is finite. \\
 	\indent   We say $E$ is a {\Ca} set in $W$ if its characteristic function $\lambda_E\in BV(W')$ for each bounded open set $W'\subset \subset W$. And 
 	$||D\lambda_E||_M(W)$ is called the perimeter of $E$ in $W$.
 \end{Def}
\br For a {\Ca} set $E$ all properities are unchanges if we make alterations of any (Lebesgue) measure zero set. Arguing exactly as Proposition 3.1 in \cite{Giu84},  we can always choose a set $E'$ differing a Hausdorff measure zero set with $E$ and satisying for any $x\in \P E'$ 
  \be \label{boundary}
  0<|E'\cap B(x,\rho)|\leq vol(B(x,\rho))
  \ene 
  where $\rho\leq \rho_0$ depending on some compact subset of $E'$ containing $x$. From now on,  we always assume that condition \eqref{boundary} holds for any {\Ca} set $E$. 
\er 

 Suppose $T$ is a $(n+1)$-dimensional  {\IMC}  in $M$, represented as  $\tau(V,\theta)$ where $V$ is a $L^{n+1}$ measurable subset of $M$. By remark \ref{rm:def:S} and the definition of the mass we have 
 \be\label{Mass_and_min}
 \bM_W(\P T)=||D\theta||_{M}(W);  
 \ene 
 for any Borel set $W\subset\subset M$.  For a derivation, see 27.7 in \cite{Sim83}. \\
 \indent There is a \emph{decomposition theorem} for codimension 1 {\IMC}s. 
 \bt [Theorem 27.6 in \cite{Sim83}]\label{dec:thm} Let $dim M=n+1$. Suppose $R$ is a $(n+1)$-dimensional {\IMC}s with $\bM_W(\P R)<\infty$ and the form $\tau(V,\theta)$.  Then there is a decreasing sequence of $L^{n+1}$ measurable sets $\{U_j\}_{j=-\infty}^\infty$ of {\Ca} sets in $M$ such that 
 \begin{gather}
 R=\sum_{j=1}^\infty [[U_j]]-\sum^0_{-\infty}[[V_j]]\quad\text {where }\quad V_j=M\backslash U_j\\
 \P R=\sum_{j=-\infty}^\infty \P [[U_j]],\quad 
 \mu_{\P  R}=\sum_{j=-\infty}^\infty \mu_{\P [[U_j]]}
 \end{gather}
 and in particular $\bM_W(\P R)=\sum_{j=-\infty}^\infty \bM_W(\P  [[U_j]])$ for $\forall W\subset\subset M$. Here $U_j:=\{x\in M: \theta\geq j\}$ for any integer $j$. 
 \et 
 \indent In the reminder of this section for a point $p$ and $r>0$ we denote the open ball centered at $p$ with radius $r$ by $B_r(p)$. Now we define an almost minimal set in an open set and a closed set respectively. 
 \begin{Def}\label{key:def}
 	Suppose $E\subset M$ is a {\Ca} set. Let $\Omega$ be a domain. 
 	\begin{enumerate} 
 		\item We say that $E$ is an almost minimal set in $\Omega$ if for any $p\in \Omega$ there is an $r_0>0$ and a constant $C$ with the property that for any $r< r_0$ and any compact set $K\subset\subset B_r(p)\subset \Omega$, 
 	\be \label{DEF:a}
 	||D\lambda_E||_{M}(B_r(p))\leq ||D\lambda_F||_M(B_r(p))+Cr^{n}
 	\ene 
 	where $F$ is any {\Ca} set satisfying $E\Delta F\subset K$. In particulark if $C=0$, we say $F$ is a minimal set in $\Omega$. 
 	\item  In (1),  if replace $\Omega$ with $\bar{\Omega}$, (the closure of $\Omega$),  we say that $E$ is an almost minimal set in $\bar{\Omega}$; 
 	\item  The regular set of $\P  E$ is the set $\{p\in \P E: \P E$ is a $C^{1,\alpha}$ graph near  $p \}$.  The singular set of $\P  E$ is the complement of the regular set in $\P E$. 
 	\end{enumerate}
 \end{Def}
\br  \label{rk:ctwo}By Lemma 7.6 in \cite{Zhou19} all $C^2$ bounded domains are almost minimal sets in an open neighborhood of their boundaries. 
\er 
 A good property of almost minimal sets in a domain is their boundary regularity.
 \bt[Theorem 1 in \cite{Tam82}, Theorem 5.6 in \cite{DS93}] Suppose a {\Ca} set $E$ is an almost minimal set in a domain $\Omega$. Let $S$ be the singular set of $\P  E$ in $\Omega$. Then 
\begin{enumerate}
	\item  if $n<7$, $S=\emptyset$;
	\item if $n=7$, $S$ consists of isolated points;
	\item if $n>7$, $H^{t}(S)=0$ for any $t>n-7$.  Here $H$ denotes the Hausdorff measure in $M$. 
\end{enumerate}
\label{regularity:accounting:thm}
 \et 
The following technique lemma would be very useful. 
 \bl [Lemma 15.1,\cite{Giu84}]\label{lm:tech} Let $W$ be an open sets and $E$, $F$ be two {\Ca} sets. Then 
 $$
 ||D\lambda_{E\cup F}||_M(W)+||D\lambda_{E\cap F}||(W)\leq ||D\lambda_E||_M(W)+||D\lambda_F||_M(W)
 $$
 \el
  Note that even the proof of Lemma 15.1 in \cite{Giu84} is given in Euclidean spaces, it works on all Riemannian manifolds.    \\
  \indent  A direct application of Lemma \ref{lm:tech} is given as follows. 
 \bl\label{lm:contain:per} Suppose $U_2\subset U_1$ be two {\Ca} sets in a Riemannian manifold $M$. Then for any open set $W\subset M$, it holds that 
 \be 
   \bM_W(\P [[U_1\backslash U_2]])\leq \bM_{W}(\P [[U_1]])+\bM_W(\P [[U_2]])
 \ene 
 \el 
 \bp Let $E$ be $U_1$ and $F$ be $U_2^c$( the complement of $U_2$).  Applying Lemma \ref{lm:tech} yields that 
 \be 
      ||D\lambda_{U_1\backslash U_2}||_M(W)\leq ||D\lambda_{U_1}||_M(W)+||D\lambda_{U_2}||_M(W)
 \ene 
 Here we use the fact that $||D\lambda_{U_2}||_M(W)=||D\lambda_{U^c_2}||_M(W)$.  The proof is complete from \eqref{Mass_and_min}. 
  \ep 
  For almost minimal sets in the closure of open domains we have a regularity result as follows. 
 \bt\label{thm:par:regularity}Let $W$ be a $C^2$  domain and $p\in \P W$. Suppose $E\subset \bar{W}$(the closure of $W$) is an almost minimal set in $\bar{W}$ and its boundary $\P E$ passes through $p$. Then $\P E$ is a $C^{1,\alpha}$ graph in an open ball containing $p$ for some $\alpha\in (0,1)$
 \et
 \bp In our proof we use $C$ to denote different constants.\\
 \indent Because $E$ is an almost minimal set in $\bar{W}$, for any $r<r^*$, any compact set $K\subset\subset  B_r(q)\subset B_{r^*}(p)$ and any {\Ca} set $F$ satisfying $E\Delta F\subset K$, we have 
  \be \label{eq:tA}
 ||D\lambda_{E}||_{M}(B_{r}(q))\leq ||D\lambda_{F\cap W}||_M(B_{r}(q))+Cr^n
 \ene
 On the other hand by remark \ref{rk:ctwo} we have 
 \be \label{eq:tB}
  ||D\lambda_{W}||_{M}(B_{r}(q))\leq ||D\lambda_{F\cup W}||_M(B_{r}(q))+Cr^n
 \ene 
 Now adding \eqref{eq:tA} and \eqref{eq:tB} together and applying lemma \ref{lm:tech}, we obtain 
       $$
        ||D\lambda_{E}||_{M}(B_{r}(q))\leq ||D\lambda_{F}||_M(B_{r}(q))+Cr^n 
       $$
   Thus $E$ is an almost minimal set in the open ball $B_{r^*}(p)$.  \\
    \indent Suppose $\P E$ passes through $p$. In terms of the local coordinate in $B_{r^*}(p)$, $\F{E-p}{\lambda}$ will converge to a minimal cone $G$ in the weak sense as $\lambda\rightarrow 0$ (for example see \cite{Tam82}).  Since $E$ is contained in a $C^2$ domain $W$, $G$ is contained in a half space of Euclidean spaces. By Theorem 15.5 in \cite{Giu84}, $G$ is equal to this half space. Thus the area density of $\P E$ at $p$ is 1. By the Allard regularity theorem,  $\P E$ is a $C^{1,\alpha}$ graph in a sufficiently small open ball containing $p$ for some $\alpha$. The proof is complete. 
 \ep 
\indent  A direct application is given as follows. 
 \bt \label{regularity:key:thm} Let $\Omega_1$ and $\Omega_2$ be two $C^2$ domains in $M$. Define $\Omega'=\Omega_1\cap \Omega_2$.  Fix a point $p$ in $\P\Omega_1\cap \P \Omega_2$.  Suppose $E\subset \Omega'$ is an almost minimal set in $\bar{\Omega}'$ (the closure of $\Omega'$) and $\P E$ passes through $p$, then $\P E$ is a $C^{1,\alpha}$ graph in an open ball containing $p$. 
 \et 
 \br  In general the boundary of $\Omega'$ is not $C^2$.
 \er 
 \bp  Let $C$ denote different positive constants. By Lemma 7.6 in \cite{Zhou19},  there is a $r^*>0$ such that $\Omega_1$ and $\Omega_2$ are almost minimal boundaries in $B_{r^*}(p)$.\\
 \indent Note that $E\subset \bar{\Omega}'$ is an almost minimal set in $\bar{\Omega}'$.  By (2) in Definition \ref{key:def},  we can choose $r^*$ sufficiently small such that for any $r< r^*$, any $q\in B_r^*(q)$, any compact set $K$ in $B_r(q)\subset  B_{r^*}(p)$ and any {\Ca} set $F$ satisfying $E\Delta F\subset K$ it holds that 
 \be \label{term:Tir}
 ||D\lambda_{E}||_{M}(B_{r}(q))\leq ||D\lambda_{F\cap \Omega_1\cap \Omega_2}||_M(B_{r}(q))+Cr^n
 \ene 
 \indent By the definition of $r^*>0$, $\Omega_1$ and $\Omega_2$ are both almost minimal sets in $B_{r^*}(p)$. If necessary we take $r^*$ small enough, it holds that 
 \begin{gather}
  \label{term:Fec}
 ||D\lambda_{\Omega_1}||_M(B_{r}(q))\leq ||D\lambda_{F\cup \Omega_1}||_M(B_{r}(q))+Cr^n\\
  \label{term:Sec}
 ||D\lambda_{\Omega_2}||_M(B_{r}(q))\leq ||D\lambda_{(F\cap \Omega_1)\cup \Omega_2}||_M(B_{r}(q))+Cr^n
\end{gather}
 Combining \eqref{term:Tir} with \eqref{term:Sec} together lemma \ref{lm:tech}  yields that 
 \be \label{term:Four}
 ||D\lambda_E||_M(B_{r}(q))\leq ||D\lambda_{F\cap \Omega_1}||_M(B_{r}(q))+Cr^n
 \ene 
 where $C$ is some constant.
 Now adding \eqref{term:Four} into \eqref{term:Fec}, lemma \ref{lm:tech} gives that  
 \be
 ||D\lambda_E||_M(B_{r}(q))\leq ||D\lambda_{F}||_M(B_{r}(q))+Cr^n
 \ene 
 From the assumption on $F$, $E$ is an almost minimal set in $B_{r^*}(p)$. \\
 \indent   Since $E\subset \Omega_1$, $\Omega_1$ is $C^{2}$ and $p\in \P E$, Theorem \ref{thm:par:regularity} implies that $\P E$ is a $C^{1,\alpha}$ graph in an open ball containing $p$. We complete the proof. 
 \ep 
  \section{The case of mean convex conformal cones}
 In this section we solve the area minimizing problem (see  \eqref{area-minimizing}) in conformal cones under reasonable conditions on $\phi(r)$ and the NCM assumption.  \\
 \indent  Let $Q_\phi$ be a conformal cone in Definition \ref{def:conformal:cone}. Let $I$ be an open interval $(-\infty, A)$ where $A$ is finite or $+\infty$. Recall that $\bar{Q}_\phi$ is the set $\bar{\Omega}\PLH I$ and $\fG$ is the set of all {\IMC}s with compact support in $\bar{Q}_\phi$. Here a closed set $F$ is compact in $\bar{\Omega}\PLH I$ means that $F\subset \Omega \PLH [a,b]$ where $[a,b]$ is a finite interval in  $(-\infty,A)$. \\
 \indent The main result in this section is stated as follows.  
  \bt \label{thm:MT:C} Suppose $\Omega$ is a bounded mean convex $C^2$ domain with the NCM property (see Definition \ref{Def:NCM}). Let $\phi(r)$ be a $C^2$ positive function satisfying 
  	\begin{gather}
  	 \tag{cA}
  	 \phi'(r)>0 \text{ on } (-\infty, A), |\F{\phi'(r)}{\phi(r)}|\leq \mu_0 \text{ on  } (-\infty,a)\\
  	 (\log\phi)^{''}\geq 0\text{ for all } r \in (-\infty,A)\tag{cB}
  	 \end{gather}
  	  for some $a<A$ and some positive constant $\mu_0$. Suppose $\psi(x)\in C^1(\P\Omega)$ satisfying $\psi(x)<A$ and $\Gamma=\{(x,\psi(x)):x\in\P\Omega \}$. Then there is a unique {\IMC} $T_0$ in $\fG$ to realize the minimum of 
\be\label{problem:label}
\min\{\bM(T): T\in \fG, \P T=\Gamma\}
\ene 
Moreover $T_0$ is equal to the graph of $u(x)$ over $\bar{\Omega}$ where $u(x)$ is the solution to \eqref{main:equation:A} with boundary data $\psi(x)$ by Theorem \ref{thm:A:B}.
 \et 
 \br This result extends the results of Rado \cite{Rado32} and Tausch \cite{Tau80} in the case of Euclidean cones. Some uniqueness results for minimal hypersurfaces will be considered in Chen-Shao-Zhou \cite{CSZ20}.
 \er 
 \br  Without the condition (cB) we can still obtain the existence of $T_0\in \fG$. But it is a question that whether $T_0$ is a graph over $\Omega$. 
\er  
Throughout this section we always assume the conditions in Theorem \ref{thm:MT:C} hold. The condition (cA) and (cB) above are just the assumptions in Theorem \ref{thm:A:B}. Thus the following definition is well-defined. 
 \begin{Def} \label{def:uk}Fix any constant $k\geq 0 $.  By Theorem \ref{thm:A:B} let $u_{-k}(x)$ be the solution to the Dirichlet problem \eqref{main:equation:A} when the boundary data is $\psi(x)- k$.
 	\end{Def}  
 This definition is used to construct a series of mean convex compact domains in $Q_\phi$ as mentioned in the introduction. From now on always suppose  $\alpha$ and $k$ are two constants satisfying $\alpha\in [\max_{\P\Omega}\psi(x), A)$ and $k\geq 0$. 
 \bl \label{mean:convex:lm}  Let $\dk$ be a domain given by 
\be\label{eq:stu} 
\dk:=\{(x,t): x\in \Omega , t\in (u_{-k}(x), \alpha)\}
\ene
and $\bdk$ be the closure of $\dk$. Then the boundary of $\dk$ is mean convex in the conformal cone $Q_\phi$ with respect to the outward normal vector except two $C^1$ submanifolds $\{(x, \alpha): x\in\P\Omega \}$ and $\{(x,u_{-k}(x)):x\in\P\Omega\}$.  Here $u_{-k}(x)$ is from Definition \ref{def:uk}. 
\el   

\bp  The boundary of $\dk$ is divided into the following four parts: 
\begin{enumerate}
	\item $A=\{(x,u_{-k}(x)):x\in\Omega\}$;
	\item $B=\{(x, \alpha):x\in \Omega\}$;
	\item $C=\{(x,t): x\in \P\Omega,  t\in (u_{-k}(x),\alpha)\}$; 
	\item $E=\{(x,\psi(x)-k):x\in\P\Omega\}\cup\{(x, \alpha): x\in\P\Omega \}$.
	 \end{enumerate}  
 \indent For part $A$, by \eqref{main:equation:A} and Corollary \ref{cor:min} its mean curvature in $Q_\phi$ is $0$. For part $B$, by Lemma \ref{lm:mc} its mean curvature with respect to the outward (upward) normal vector is $n\F{\phi'(\alpha)}{\phi^2(\alpha)}$. It is positive due to the condition (cA).\\
 \indent  As for part $C$,  its normal vector is perpendicular to $\P_r$. Let  $\tilde{H}_C$ be the mean curvature of $C$ with respect to the outward normal vector in $Q_\phi$.  By Lemma \ref{lm:mean:curvature},  $\tilde{H}_C$ is $\F{1}{\phi(r)}H_{\P\Omega}$ where $H_{\P\Omega}$ is the mean curvature of $\P\Omega$. Since $\Omega$ is mean convex, $\tilde{H}_C\geq 0$ with respect to the outward nomral vector. \\
\indent In summary the mean curvature of part $A,B$ and $C$ in $\dk$ are nonnegative with respect to the outward normal vector of $\dk$. The proof is complete. \ep 
 We consider a local version of the area minimizing problem in \eqref{problem:label} as follows. 
 \be
A_{k,\alpha}:=\min \{\bM(T): spt T\subset \bdk, \quad  \P T=\Gamma\}
 \ene
 By Theorem \ref{compact:thm} there is an {\IMC} $T_{k,\alpha}$ contained in  $\bdk$ with $\P T_{k,\alpha}=\Gamma$ satisfying $\bM(T_{k,\alpha})=A_{k,\alpha}$. 
 \bl \label{eq:lm:st} Fix $\alpha\in (\max_{\P\Omega}|\psi(x)|, A)$ and $k > 0$.  Let $\tka$ be given as above. Suppose $\Omega$ is contained in a complete {\RM} $N$. Then  there is a {\Ca} set $F$ in  $N\PLH (-\infty, A)$ such that $T_{k,\alpha}=\P [[F]]|_{\bdk}$. 
 \el
 \br \label{remark:note} The conclusion of the above lemma still holds if the domain $D_{k,\alpha}$ is replaced with any open set $\Omega\PLH (a,b)$ in $Q_\phi$ where $[a,b]\subset (-\infty, A)$. Our proof is inspired from that of Lemma 7 in \cite{Ilm96}.  
 \er 
 \bp Let $v(x)$ be a $C^2$ function on $N$ such that $u_{-k}(x)< v(x) < \alpha$ on $\bar{\Omega}$ and $v(x)=\psi(x)$ on $\P\Omega$. Let $E$ be the subgraph of $v(x)$ in $N\PLH\R$. That is $\{(x,t):x\in N, t<v(x)\}$. Define $S:=\P [[E]]|_{\bdk}$. Thus $\P S=\Gamma$ and $
 \P(T_{k,\alpha}-S)=0$.\\
 \indent  By Theorem \ref{thm:constant:current} there is a $n+1$ {\IMC} $R$ in  $N\PLH (-\infty, A)$ such that $T_{k,\alpha}-S=\P  R$. Then we have 
 \be\label{eq:reason:first}
 T_{k,\alpha}=\P [[E]]|_{\bdk}+\P R
\ene 
 Observe that $[[E]]+R$ can be represented as $\tau(N\PLH (-\infty,A), \theta)$ where $\theta$ is some integer value measurable function on $N\PLH \R$. Now define a function $\theta_1=\theta$ if $\theta\neq 1$ and $\theta_1=0$ if $\theta=1$. Let $\theta_0$ denote the function of $\theta-\theta_1$. \\
 \indent Set $F:=\{p\in N\PLH(-\infty,A): \theta(p)=1\}$. It is not hard to see that $[[F]]=\tau(N\PLH(-\infty, A), \theta_0)$. For notation see remark \ref{rk:mark:notation}. As a result one has 
    \be \label{eq:reason:two}
   [[E]]+R=[[F]]+G
    \ene 
    where $G$ is the {\IMC} $\tau(N\PLH (-\infty, A),\theta_1)$. From the definition of $E$ and $R$,  we have $spt(G)\subset \bar{Q}_\phi$ which is the set $ \bar{\Omega}\PLH (-\infty,A)$. \\
    \indent Now define  $U_j=\{ p\in N\PLH (-\infty,A): \theta \geq j\}$ for any integer $j$. By the definition of $E$, we have $spt(\P[[U_j]])\subset \bar{Q}_\phi$ for any $j\neq 1$ . Note that $spt(T_{k,\alpha}) \subset \bdk \subset \bar{Q}_\phi$.  Applying the decomposition theorem (Theorem \ref{dec:thm}) on $T_{k,\alpha}$ we obtain 
   \begin{align*}
  \mu_{ T_{k,\alpha}}&=\sum_{j=-\infty,j\neq 1}^\infty\mu_{\P [[U_j]]}+\mu_{\P  [[U_1]]}|_{\bar{Q}_\phi}\\
   &=\sum_{j=-\infty,j\neq 1}^\infty\mu_{\P[[ U_j]]}|_{\bdk}+\mu_{\P  [[U_1]]}|_{\bdk}
   \end{align*}
    This implies that $spt(\mu_{\P [[U_j]]})\subset \bdk$ for each $j\neq 1$ and $spt(\mu_{\P[[ U_1]]})|_{\bar{Q}_\phi}\subset \bdk$.  For $G=\tau(M,\theta_1)$ applying the decomposition theorem gives that 
       \be 
    \mu_{\P G}=\sum_{j=3}^\infty \mu_{\P[[ U_j]]}+2\mu_{\P [[U_2]]}+\sum_{j=-\infty}^0 \mu_{\P [[U_j]]}
    \ene 
    Thus $spt(\P  G)\subset \bdk$.\\
    \indent  As for $F$ the decomposition theorem gives that $\P [[F]]=\P [[U_1\backslash U_2]]$. Since $U_2\subset U_1$, with \eqref{eq:reason:first} and \eqref{eq:reason:two}, Lemma \ref{lm:contain:per} implies that 
     $$
     \mu_{\P [[F]]}|_{\bdk}\leq \mu_{\P [[U_1]]}|_{\bdk}+ \mu_{\P [[U_2]]} \leq \mu_{T_{k,\alpha}}
     $$ 
     In particular $\bM(\P [[F]]|_{\bdk})< \bM(T_{k,\alpha})<\infty$ unless $\mu_{\P G}= 0$. Due to the minimality of $\bM(T_{k,\alpha})$ we have $\P G=0$ and $T_{k,\alpha}=\P [[F]]|_{\bdk}$. \\
     \indent The proof is complete.  \ep 
 Next we show that $T_{k,\alpha}$ only touches the boundary of $\bdk$ at $\Gamma=(x,\psi(x))$. 
 \bl \label{lm:disjoint} With the notaton in \eqref{eq:stu}, $\tka$ is disjoint with $\P\dk\backslash \Gamma$ for any  $k>0$ and any $\alpha \in (\max_{\P\Omega}\psi(x), A)$. 
 \el 
 \bp We argue it by contradiction. 
 Now suppose there is a point $p$ in $(\P \dk \backslash \Gamma)\cap \tka $. \\
 \indent No matter where the position of $p$ is, near $p$ we can view $\P\dk$ is the intersection of two $C^2$ boundaries.  On the other hand from the definition of $\tka=\P [[F]]|_{\bdk}$, $F$ is a minimal set in  the closed set $\bdk$.  Since $\tka$ contains $p$, $\P F$ passes  through $p$. By Theorem \ref{regularity:key:thm}  $\P F$ is a $C^{1,\alpha}$ graph near $p$ (an open ball containing $p$). Moreover it is easy to see that
 \be  \label{de}
  H_{\P F}\leq 0
 \ene 
 near $p$ with respect to the outward normal vector of $\dk$ in the Lipschitz sense.  \\
 \indent By Lemma \ref{mean:convex:lm} the boundary of $\dk$ is mean convex.   Since $\Omega$ is $C^2$ and mean convex, $H_{\P\Omega\PLH\R}\geq 0$ with respect to the outward normal vector of $\Omega\PLH\R$.  By Theorem \ref{max:thm:one} and \eqref{de}, $\P\Omega\PLH\R$ coincides with $\P F|_{\bdk}$ near $p$. Because $k>0$ and $\alpha \in  (\max_{\P\Omega}\psi(x),A)$, the set $
 \{(x,\psi(x)-k)\cup (x, \alpha): x\in\P\Omega \}$ is disjoint with $\Gamma$. By the connectedness of $\P\Omega\PLH \R$, there is at least one point $p'\in \P F_{\bdk}$ in the set $
 \{(x,\psi(x)-k)\cup (x, \alpha): x\in\P\Omega \}$. Moreover the tangent sapce of $\P\Omega\PLH\R$ at $p'$ is equal to that of $\P F$ at $p'$.  Then there is a tangent vector of $\P F$ at $p'$ pointing outward with respect to $\bdk$. Because $p'\notin \Gamma$, all tangent vectors of $\P F$ at $p'$ should point into $\bdk$. Otherwise $\P (\P F|_{\bdk} )$ is not equal to $\Gamma$.  This is a contradiction.  \\
 \indent Thus $\tka$ is disjoint with $\P \dk\backslash \Gamma$. The proof is complete.
 \ep 
 A  direct application of the above lemma is 
 \begin{cor}\label{de:cor}  With the notaton in \eqref{eq:stu}, $\tka\subset \bar{D}_{0,\alpha_0}\backslash\{(x,\alpha_0):x\in \Omega\}$ where $\alpha_0=\max_{\P\Omega}\psi(x)$ and  \be
 	D_{0,\alpha_0}:=\{(x,t): x\in \Omega, t\in (u_0(x), \alpha_0)\}
 	\ene
 	Here $u_0(x)$ is from Definition \ref{def:uk}. 
 	\end{cor}
 \bp  From \eqref{main:equation:A} and Definition \ref{def:uk}, the condition (cB) implies that $u_{k}(x)\leq u_{k'}(x)$ if $k>k'\geq 0$.  By Lemma \ref{lm:disjoint}, there is a $k'\in (0,k)$ and some $\alpha'\in (\alpha_0, A)$ such that  $T_{k,\alpha}$ is disjoint with $\P D_{k',\alpha'}\backslash \Gamma$.. Repeating this process $T_{k,\alpha}$ is always disjoint with $D_{k',\alpha'}$ for all $k'\in (0,k)$ and all $\alpha'\in (\alpha_0, A)$.  Taking the intersections of those domains, we obatin $T_{k,\alpha}\subset \bar{D}_{0,\alpha_0}$. Since $\{(x,\alpha_0):x \in \Omega\}$ is mean convex with respect to the outward normal vector of $D_{0,\alpha_0}$, arguing as in Lemma \ref{lm:disjoint}, $\tka$ is disjoint with $\{(x,\alpha_0):x \in \Omega\}$. The proof is complete. 
 \ep 
 \begin{Def}Now we define a new $C^2$ positive function $\tps(r)$ on $(-\infty, A)$ as 
	\be \label{def:tps}
\tps(r)=\left\{\begin{split}
	& C e^{\beta r}\quad r >\F{A+\alpha_0}{2}\\
	        &\phi(r)\quad  r<\alpha_0
\end{split}\right. 
\ene 
Here $\beta$ is $\max\{r\in(-\infty, \F{A+\alpha}{2}):\F{\phi'(r)}{\phi(r)}\}$ and $C>0$ is a large constant such that $\tps'(r)>0$ on $[\alpha_0, A]$ and $(\log \tps)''(r)\geq 0$ on $(-\infty,A)$. Here $\alpha_0=\max_{\P\Omega}\psi(x)$. 
\end{Def}
\br\label{rk:mst}
Since $\phi(r)$ satisfies the condition (cA) and (cB) in Theorem \ref{thm:A:B},  so is $\tps(r)$. \\
\indent Moreover by the definition of $\tps(r)$ there is a constant $\beta_0>\beta>0$ such that 
$|\F{\tps'(r)}{\tps(r)}|\leq \beta_0$ for all $r\in (-\infty,A)$. 
 \er 
\indent Recall that in this section $\Omega$ is bounded, mean convex with the NCM property.  By the above remark and Theorem \ref{thm:A:B} the following definition is well-defined. 
\begin{Def}\label{def:mu} For any $\mu\geq 0$ let $\tilde{u}_{\mu}(x)$ be the solution to the Dirichlet problem
	 \begin{equation}\label{def:rstq}
	\left\{\begin{split}
	div(\frac{Du}{\omega}) &=\frac{n\tps^{'}(u(x))}{\tps(u(x))\omega} \quad x\in \Omega \quad \omega =\sqrt{1+|Du|^2}\\
	u(x)&=\psi(x)+\mu \quad x\in \partial\Omega
	\end{split}\right . 
	\end{equation}
	where $\psi(x)$ is the $C^1$ function given in the assumptions of Theorem \ref{thm:MT:C}. 
	\end{Def}
  Thus combining Lemma \ref{lm:continous} with remark \ref{rk:mst} $\tilde{u}_\mu(x)$ is continous with respect to $\mu\in [0,+\infty)$.  Moreover
 \bl \label{lm:compare}Let $\tilde{u}_{\mu}(x)$ be the family of smooth functions given in Definition \ref{def:mu}. Then there is a $\mu_0>0$ such that $\tilde{u}_{\mu_0}(x)\geq \alpha_0$ on $\Omega$. 
 \el 
 \bp Let $\beta_0$ be the constant given in remark \ref{rk:mst}. Let $v_\mu(x)$ be the solution to the following problem  
  \begin{equation}
 \left\{\begin{split}
 div(\frac{Du}{\omega}) &=n\frac{\beta_0}{\omega} \quad x\in \Omega \quad \omega =\sqrt{1+|Du|^2}\\
 u(x)&=\psi(x)+\mu\quad x\in \partial\Omega
 \end{split}\right . 
 \end{equation}
 By the maximum principle and the definition of $\beta_0$, we have $v_\mu(x)\leq \tilde{u}_\mu(x)$ for each $x\in\bar{\Omega}$ and each $\mu>0$.  The conclusion follows from $v_\mu(x)=v_0(x)+\mu$. 
 \ep 
 Now we are ready to conclude Theorem \ref{thm:MT:C} from Corollary \ref{de:cor}.
 \bp [The proof of Theorem \ref{thm:MT:C}] According to Corollary \ref{de:cor},  $\tka\subset \bar{D}_{0,\alpha_0}$. For each $\mu\in [0,\mu_0]$ define a function $\kappa_\mu(x)=\min \{\alpha_0,\tilde{u}_\mu(x)\}$. By lemma \ref{lm:compare} at $\mu_0$ we have $\kappa_{\mu_0}(x)=\alpha_0$. \\
 \indent  Next we consider the value of $\mu$ when the graph of $k_\mu(x)$ firstly touches $T_{k,\alpha}$ given by 
 \be 
 a:=\inf\{s\in (0, \mu_0]: T_{k,\alpha} \cap gra(\kappa_t(x))=\emptyset\text{ for all $t\in [s,\mu_0]$ } \}
 \ene
 where $gra(\kappa_s)(x)$ is the graph of $\kappa_s(x)$.  \\
 \indent By corollary \ref{de:cor}, $T_{k,\alpha}$ is disjoint with the upper boundary of $\P D_{0,\alpha_0}$, i.e. $\{(x,\alpha_0):x\in\P\Omega\}$. By the continuity of $\kappa_{\mu}(x)$ we conclude $a<\mu_0$.  \\
 \indent Suppose $a>0$.  Then $T_{k,\alpha}$  has to touch $gra(\kappa_a(x)$ at some point $(x_0,t_0)$ with $t_0<\alpha_0$.  This implies that in a small neighborhood of $(x_0,t_0)$,  $\kappa_a(x)=\tilde{u}_a(x)$. By \eqref{def:tps} and Definition \ref{def:mu}  $gra(\kappa_a(x))$ near $p$ is minimal in $Q_\phi$. Combining \eqref{de} and Theorem \ref{max:thm:one} together, $T_{k,\alpha }$ coincides with $gra(\kappa_a)$ near $p$.  Due to the connectedness of $gra(\kappa_a)$, then $T_{k,\alpha}$ has to touch some point on $\{(x,\alpha_0):x\in\Omega\}$. This is a contradiction. Then $a=0$.\\
 \indent By \eqref{def:rstq} and $\tps'(r)>0$ for all $r\in \R$, $\tilde{u}_0(x)\leq \max_{\Omega}\psi(x)=\alpha_0$. As a result $\kappa_0(x)=\tilde{u}_0(x)$.  By \eqref{def:tps} and \eqref{def:rstq}, $\tilde{u}_0(x)$ solves the Dirichlet problem 
  \begin{equation}
 \left\{\begin{split}
 div(\frac{Du}{\omega}) &=\frac{n\phi'(u(x))}{\phi(u(x))\omega} \quad x\in \Omega \\
 u(x)&=\psi(x)\quad x\in \partial\Omega
 \end{split}\right . 
 \end{equation}
 This is the Dirichlet problem  \eqref{main:equation:A} with boundary data $\psi(x)$. By the condition (cB) the above equation has a unique solution. Together with Definition \ref{def:uk}, we conclude that $\tilde{u}_0(x)=u_0(x)$.  On the other hand, since $a=0$, 
  $T_{k,\alpha}$ is contained in a domain 
 \be 
  \{(x,t): u_0(x) \leq t \leq \kappa_0(x)=\tilde{u}_0(x)\}
 \ene 
 Thus $T_{k,\alpha}(x)$ is just the graph of $u_0(x)$ over $\bar{\Omega}$.\\
 \indent  Let $T_0$ be the graph of $u_0(x)$. Now let $k\rightarrow +\infty$ and $\alpha\rightarrow A$. Note that for any {\IMC} $T\in \fG$ and $\P T=\Gamma$ there is some $k>0,\alpha\in (0, A)$ such that $spt(T)\subset \bar{D}_{k,\alpha}$. Thus 
 $\bM(T_0)=\bM(T_{k,\alpha})\leq \bM(T)$.  Thus $T_0$ realizes the minimum of 
 $$
 \{\bM(T): T\in \fG, \P T=\Gamma\} 
 $$
 The proof is complete. 
  \ep

\section{The translating conformal cone}\label{section:translating}
In the previous section we require that the conformal cone is mean convex.  In this section we remove this condition in the case of translating conformal cones and consider the area minimizing problem (see Theorem \ref{final:conclusion:thm}). Our main tool is the BV function theory from \cite{Giu84}, \cite{Lin85} and \cite{Zhou19}.  We extend Lin's result \cite{Lin85} in clyinders of product manifolds into translating conformal cones. 
\subsection{Area functional} Let $\alpha>0$ be a fixed constant. Throughout this section let $\Omega$ be a $C^2$ bounded domain in a $n$-dimensional manifold $N$ with a metric $\sigma$. From Definition \ref{def:conformal:cone}, we single out the following concept. 
\begin{Def}  The translating conformal cone is defined by 
$$
Q_\alpha:=\{\Omega\PLH\R,  e^{2\alpha\F{r}{n}}(\sigma+dr^2)\}
$$

\end{Def}

\indent We use a similar setting as in the previous section. Let $\psi(x)$ be a $C^1$ function on $\P\Omega$ and $\Gamma=(x,\psi(x))$.  Define $\fG$ as the set of all {\IMC}s  in $\bar{Q}_\alpha$ of which support is contained in some set $\bar{\Omega}\PLH [a,b]$ with two finite constants $a<b$. The area minimizing problem in this section is still to find an {\IMC} $T_0\in \fG$ to realize 
\be\label{am:problem}
\min \{\bM(T): T\in\fG,    \P T=\Gamma\}
\ene 
\br The area minizing problem in the case of $\alpha=0$, i.e.  the product manifold, was investigated in \cite{LL85} and \cite{Lin85}.
\er 
\indent Now we consider a conformal functional on BV functions. For preliminary facts on BV functions, see section 3. \begin{Def}  \label{Def:conformal:functional}Suppose $u\in BV(\Omega)$ and $\alpha>0$.  Define
	\be 
	\begin{split}
		\AF_\alpha(u,\Omega):=\sup\{&\int_\Omega  e^{\alpha u(x)}(h+\F{1}{\alpha}div(X))dvol:\\
		&h\in C_0(\Omega), X\in T_0\Omega,  h^2+\la X,X\ra\leq1
		\}
	\end{split}
	\ene 
	where $\la,\ra$, div and $d vol$ are the inner product, the divergence and the volume form of $\Omega$ respectively and $C_0(\Omega)$ $(T_0(\Omega))$ denotes the set of all continuous functions ($C^1$ differential vector fields ) with compact support in $\Omega$. 
\end{Def}
\br  If $u\in C^1(\Omega)$, then $\AF_\alpha(\Omega,u)=\int_\Omega e^{\alpha u(x)}\sqrt{1+|Du|^2}d vol$ and is equal to the area of the graph of $u(x)$ in $Q_\alpha$. It was firstly studied in \cite{Zhou19}. 
\er 
\begin{Def}
For any function $u(x)$ on $N$ the set $U=\{(x,t):x\in \Omega, t<u(x)\}$ is called as the subgraph of $u(x)$ in $\Omega\PLH\R$.
\end{Def}
From now on assume that $\Omega$ is contained in a larger $C^2$ domain $\Omega^*$.  Consider the following minimizing problem 
\be \label{BV:minimum}
\min\{ \AF_\alpha(u,\Omega^*): u\in BV(\Omega^*)\text{ bounded, }  u=\psi(x)\text{ outside }\Omega\}
\ene
The relationship between the above problem and \eqref{am:problem} is stated as follows. 
\bt\label{central:fact}  Let $\Omega$, $\Omega^*$ be domains given as above. Suppose $\psi(x)\in C^1(\Omega^*\backslash \Omega)$. Let $u(x)\in BV(\Omega^*)$ be a bounded function with $u=\psi(x)$ outside $\Omega$.  Define $
  T:=\P [[U]]|_{\bar{\Omega}\PLH\R}$
  where $U$ is the subgraph of $u(x)$. 
Then $T$ is the solution to the problem in \eqref{am:problem} if and only if $u(x)$ is the solution to the problem in \eqref{BV:minimum}. 
\et 
\br The case of $\alpha=0$ of this theorem was firstly observed by Lau-Lin \cite{LL85}. See also the introduction of \cite{Bourni11}.  \\
\indent Let $\tu(x)$ be the trace of $u(x)$ on $\P\Omega$ from $\Omega$. For the definition of the trace see remark 2.5 in \cite{Giu84}. It may not be equal to $\psi(x)$ in general. Let $T$ be $gra(u(x))+S$  where $S$ is the $n$-dimensional {\IMC} in $\P \Omega\PLH\R$ enclosed by $\psi(x)$ and $\tu(x)$ such that  $\P S=\Gamma-gra(\tu(x))$. Moreover $\P T=\Gamma$. The mass of $T$ restricted in $\bar{Q}_\alpha$(the set $\bar{\Omega}\PLH\R)$ is 
\be \label{eq:BV:minimum}
\bM(T)=\AF_\alpha(u,\Omega)+\int_{\P \Omega}\F{1}{\alpha}|e^{\alpha \tu(x)}-e^{\alpha \psi(x)}|dvol
\ene 
\er 
\subsection{Some preliminary facts} Now we will collect some preliminiary facts for the proof of Theorem \ref{central:fact}.  In \cite{Zhou19} we show many results on the connection between BV functions and the conformal functional $\AF_\alpha(u,\Omega)$. \\
\indent In this subsection suppose $W$ is a fixed open set in $\Omega$. 
\bt [Theorem 5.5 in \cite{Zhou19}]  \label{thm:st:uv}Let $u(x)\in BV(W)$ and $U$ be its subgraph. Then it holds that 
\be 
||D\lambda_U||_{Q_\alpha}(W\PLH\R)=\AF_\alpha(u,W)
\ene 
for any $\alpha\geq 0$.\label{perimeter}
\et 
The following results are easily obtained from the definition of $\AF_\alpha(u,\Omega)$. 
\bl  [Corollary 5.7 in \cite{Zhou19}]\label{det:pro} Use the notation in Definition \ref{Def:conformal:functional}.  It holds that 
\begin{enumerate} 
	\item  Suppose $|u|\leq k$ on $W$, $\AF_\alpha(u,W)\geq e^{-\alpha k }(||Du||_N(W)+vol(W))$;
	\item  Suppose $\{u_j\}_{j=1}^\infty\in BV(W)$ are uniformly bounded and converges to
	$u$ in $L^1(W)$, then  $u\in BV(W)$ and  $$\AF_{\alpha}(u,W)\leq \lim_{j\rightarrow\infty} \inf \AF_\alpha(u_j,W)$$
\end{enumerate}
\el
The following result is a minor modification of Lemma 5.8 in \cite{Zhou19}. 
\bt \label{thm:est}
Let $u_1(x)<u_2(x)$ be two bounded BV functions on $W$.  Let $F$ be any {\Ca} set satisfying $\P F\subset W\PLH (u_1(x),u_2(x))$.  Then there is a function $w(x)\in BV (W)$ such that 
\be\label{dt:con:thm}
||D\lambda_F||_{Q_\alpha}(W\PLH\R)\geq \AF_\alpha(w(x),W)
\ene 
with the property that $u_1(x)\leq \omega(x)\leq u_2(x)$. Here $\alpha >0$. 
\et 
\br\label{reason:remark}  This theorem is the reason that here we only consider translating conformal cones.  For general conformal cones it is unknown whether a similar result as above holds. 
\er 
\bp  By Lemma 5.8 in \cite{Zhou19}, there is a function $\omega(x)\in BV(W)$ such that \eqref{dt:con:thm} holds. Moreover $\omega(x)$ is
defined by 
\be \label{rt:st}
e^{\alpha \omega(x)}=\alpha\lim_{k\rightarrow+\infty}\int_{-k}^k e^{\alpha t}\lambda_F(x,t)dt 
\ene  
Let $U_1$ and $U_2$ be the subgraph of $u_1(x)$ and $u_2(x)$ respectively. Because $\P F\subset W\PLH (u_1(x), u_2(x))$, then $\lambda_{U_1}(x,t)\leq \lambda_F(x,t)\leq \lambda_{U_2}(x,t)$. Combining this with \eqref{rt:st}, we obtain $e^{\alpha u_1(x)}\leq e^{\alpha \omega(x)}\leq  e^{\alpha u_2(x)}$. The positivity of $\alpha$ implies that $u_1(x)\leq \omega(x)\leq u_2(x)$. 
\ep 
The Miranda observation says that if $u(x)$ is a local minimum of $\AF_\alpha(.,W)$ then its subgraph has local least perimeter. A precise statement is given as follows. 
\bt [The Miranda's observation]\label{obs:thm} Let $u_1(x)$ and $u_2(x)$ be two measurable functions which may take possible infinity values.  Let $\tilde{W}$ be an open set taking the form $W \PLH (u_1(x), u_2(x))$ in $W\PLH \R$. Suppose $u(x)\in BV(W)$ with the following property: for any $v(x)\in BV(W)$ with its subgraph $V$ satisfying $V\Delta U\subset K$ for some compact set $K$ in $\tilde{W}$, it holds that
\be\label{eq:condition:A}
\AF_\alpha(u(x),W)\leq \AF_\alpha(v(x),W) 
\ene 
 Then its subgraph $U$ satisfies that 
 \be \label{conclusion:C}
 ||D\lambda_U||_{Q_\alpha}(\tilde{W})\leq ||D\lambda_F||_{Q_\alpha}(\tilde{W})
 \ene 
 for any {\Ca} set $F$ satisfying $F\Delta U$ containing some compact set in $\tilde{W}$.  The same conclusion holds if replace $\tilde{W}$ with its closure. 
\et

\bp  Fix any compact set $K$ in $\tilde{W}$.  Let $F$ be any {\Ca} set satisfying $U\Delta F\subset K$. There is a constant $\Sc>0$ such that $K\subset W\PLH (u_1(x)+\Sc, u_2(x)-\Sc)$. By Theorem \ref{thm:est}, there is a $\omega(x)$ such that 
\be 
\AF_\alpha(\omega(x), W)\leq ||D\lambda_F||_{Q_\alpha}(W\PLH\R)
\ene 
where $u_1(x)+\Sc\leq \omega(x)\leq u_2(x)-\Sc$. From the definition of $\omega(x)$ in \eqref{rt:st}, 
the subgraph of $\omega(x)$, $V$, satisfies $V\Delta U\subset K'$ where $K'$ is compact dertermined by $K$ and $\Sc$.  By Theorem \ref{perimeter} and \eqref{eq:condition:A}
\begin{align*}
||D\lambda_U||_{Q_\alpha}(W\PLH\R)&\leq \AF_\alpha(\omega(x),W)\\
&\leq ||D\lambda_F||_{Q_\alpha}(W\PLH\R)
\end{align*}
\indent Since $K$ is compact in $\tilde{W}$, we obtain the conlusion \eqref{conclusion:C}. As for the closure case, everything proceeds exactly except $u_1(x)<\omega(x)<u_2(x)$.  We can use $u_1(x)\leq \omega\leq u_2(x)$.  The proof is complete. 
\ep 
A BV function to achieve a local minimum of $\AF_\alpha(u,\Omega)$ in open set has some good regulairty.  An equivalent statement of Theorem 8.3 in \cite{Zhou19} says that 
\bt \label{regularity:theorem} Let $W$ be an open set.  Suppose $u(x)$  uniformly bounded satisfies 
$
 \AF_\alpha(u, W)\leq \AF_\alpha(u,v(x))
$
for any $V$ with the property that $U\Delta V$ is contained in a compact set in $\bar{\Omega}\PLH\R$ where $U,V$ are subgraphs of $u(x)$ and $v(x)$ respectively. Then $u(x)\in C^2(W)$. 
\et 

\subsection{The proof of Theorem \ref{central:fact}} 
\bp ``$\Rightarrow$''  Suppose $T=\P [[U]]|_{\bar{\Omega}\PLH\R}$ is the solution to the problem in \eqref{am:problem}. Here $U$ is the subgraph of $u(x)$. Fix any bounded function $v\in BV(\Omega^*)$ satisfying $v(x)=\psi(x)$ outside $\Omega$. \\
\indent Let $V$ be its subgraph.  Consider $T'=\P [[V]]|_{\bar{\Omega}\PLH\R}$. Then $\P T'=\Gamma$ and $spt(T')\subset \subset \bar{\Omega}\PLH\R$. By \eqref{am:problem}, we have 
\be 
\bM(T)\leq \bM (T')
\ene 
Since $U=V$ outside $\bar{\Omega}\PLH\R$, we obtain that 
\be 
\bM(\P [[U]])\leq \bM (\P [[V]])
\ene
Here the support of $\P [[U]]$ and $\P [[V]]$ are located in $\Omega^*\PLH\R$. 
By \eqref{Mass_and_min}, the above inequality is equivalent to 
\be 
||D\lambda_U||(\Omega^*\PLH\R)\leq  ||D\lambda_V||(\Omega^*\PLH\R)
\ene 
By Theorem \ref{thm:st:uv} one obtains that 
\be \label{key:les:st}
  \AF_\alpha(u,\Omega^*)\leq \AF_\alpha(v,\Omega^*)
\ene
From the way to choose $v(x)$ $u(x)$ is the solution to the problem in \eqref{BV:minimum}. \\
\indent ``$\Leftarrow$"  Suppose $u(x)\in BV(\Omega^*)$ is bounded as the solution to the problem in \eqref{BV:minimum}. Let $T=\P [[U]]|_{\bar{\Omega}\PLH\R} $ where $U$ is the subgraph of $u(x)$.  Fix any $T^*$ satisfing $spt(T^*)\subset \subset \bar{\Omega}\PLH\R$ with $\P T^*=\Gamma$. We can assume $spt(T^*)$ and $spt(T)$ are contained in $ \bar{\Omega}\PLH (-a,a)$ for sufficiently large $a$. Let $T_a$ be an {\IMC} to achieve
\be
\min \{\bM(T): spt T\subset \subset \bar{\Omega}\PLH [-a,a]  \quad  \P T=\Gamma\}
\ene 
By Lemma \ref{eq:lm:st} and Remark \ref{remark:note}, there is a {\Ca} set $F$ such that $T_a=\P [[F]]|_{\bar{\Omega}\PLH\R}$.  Thus 
\be 
   ||D\lambda_F||(\bar{\Omega}\PLH (-a,a))=\bM(T_a)\leq \bM(T^*)
   \ene 
  Now define a new {\Ca} set $F'$ such that $F'$ coicides with $U$ outside $\Omega^*\backslash \bar{\Omega}\PLH\R$ and $F'$ coincides $F$ in $\bar{\Omega}\PLH\R$.  By Theorem \ref{thm:est}, there is $v(x)\in BV(\Omega^*)$ such that 
$$ 
 \AF_\alpha(v(x),\Omega^*) \leq ||D\lambda_{F'}||(\Omega^*\PLH\R)=\bM(\P [[F']])
$$
with $v(x)=\psi(x)$ outside $\Omega$. Thus by \eqref{BV:minimum}
\be 
\bM(\P [[U]])=\AF_\alpha(u,\Omega^*)\leq \bM(\P [[F']])
\ene
Because $U$ concides with $F'$ outside $\bar{\Omega}\PLH\R$, we have 
\be
M(T)\leq  \bM(T_a)\leq \bM(T^*)
\ene 
Thus $T$ is the solution to the problem in \eqref{am:problem}. 
\ep 
Now a direct application of Theorem \ref{central:fact} and Theorem \ref{thm:MT:C} is concluded as follows.  
\bt\label{minim:thm}Fix $\Omega \subset \subset \Omega^*$ be two bounded $C^2$ domains. Suppose 
\begin{enumerate}
	\item $\Omega$ is a $C^2$ mean convex, bounded domain with the NCM property. 
	\item $\psi(x) \in C^1(\Omega^*\backslash \Omega)$.  
\end{enumerate}
	Let $u_\infty(x)$ be the solution of 
	\be 
	\left\{\begin{split}
		div(\frac{Du}{\omega}) &=\frac{\alpha}{\omega} \quad x\in \Omega\quad \omega =\sqrt{1+|Du|^2}\\
		u(x)&=\psi(x)\quad x\in \partial\Omega
	\end{split}\right . 
	\ene 
Let $u_\infty(x)=\psi(x)$ outside $\Omega$. Then $u_\infty(x)$ achieves the minimum in 
 \be 
 \min\{ \AF_\alpha(u,\Omega^*): u\in BV(\Omega^*)\text{ bounded},  u=\psi(x) \text{ outside }  \Omega\}
 \ene 
\et  
\subsection{The main result}
Now we state the main result of this section.  
\bt \label{final:conclusion:thm} Let $\Omega\subset\subset \Omega^*$ be two $C^2$ bounded domains. Suppose $\Omega^*$ has the $C^2$ mean convex boundary with the NCM property and $\psi(x)\in C^1(\Omega^*\backslash \Omega)$. Let $\Gamma=(x,\psi(x))$. Then  \begin{enumerate}
	\item there is a bounded BV function $u_\infty(x)$ with $u_\infty(x)=\psi(x)$ outside $\Omega$ to achieve the minimum of 
	\be 
	 \min\{ \AF_\alpha(u,\Omega^*): u\in BV(\Omega^*),  \text{bounded}, u=\psi(x)\text{ outside }\Omega\}
	\ene 
	\item Let $U_\infty$ be the subgraph of $u_\infty(x)$. Then $T_\infty=\P [[U_\infty]]|_{\bar{\Omega}\PLH\R}$  is an {\IMC} to  solve the problem in \eqref{am:problem}，i.e.
	\be
	\bM(T_\infty)=\min \{\bM(T): T\in \fG,  \quad  \P T=\Gamma\}	
	\ene
	\item $u_\infty(x)$ is $C^2$ on $\Omega$. 
\end{enumerate}
\et 
\bp 
By Theorem \ref{central:fact} the conclusion (1) is equivalent to the conclusion (2). The conclusion (3) is from Theorem \ref{regularity:theorem}.  Thus we only need to prove the conclusion (1). \\
\indent By Theorem \ref{tsh:thm} let $v(x)$ be the solution to the following equation 
\be \label{min:eq}
\left\{\begin{split}
	div(\frac{Du}{\omega}) &=\frac{\alpha}{\omega} \quad x\in \Omega^*\quad \omega =\sqrt{1+|Du|^2}\\
	u(x)&=1\quad x\in \partial\Omega^*
\end{split}\right . 
\ene 
Because $\psi(x)$ is $C^1$ on $\P\Omega$, there is a large positive integer $k$ such that $$\max_{x\in \P\Omega}|\psi(x)|\leq k$$
Let $\mu_0=\max\{k+1, \max_{\bar{\Omega}*} (v(x)-k)\}$.  Fix any $\mu>\mu_0$. Now we consider the following auxiliary problem 
\be\label{au:BV:minimum}
\begin{split}
A_\mu:=\{ \AF_\alpha(u,&\Omega^*): u\in BV(\Omega^*),
v(x)-\mu\leq u\leq v(x)+\mu \text{ on }  \Omega  \\
&u=\psi(x)\text{ on }\Omega^*\backslash\Omega\}
\end{split}
\ene 
By the definition of $\mu_0$, the above definition is well-defined. Note that on $\P\Omega$, $v(x)-\mu \leq \psi(x)\leq v(x)+\mu$.  Suppose $\{u_j\}$ is a sequence of BV functions satisfying \eqref{au:BV:minimum} and $\lim_{j\rightarrow +\infty} \AF_\alpha(u_j,\Omega)=A_\mu $.  Moreover $|u_j|\leq C$ for all $j$ where $C$ is a positive constant depending on $\mu$ and $\psi(x)$.  By (1) in  Lemma \ref{det:pro}, we have 
\be 
\{||Du_j||_N(\Omega^*)+vol(\Omega^*) \}
\ene 
is uniformly bounded.  By the compactness of BV functions on Lipschitz domains, there is $u_\mu(x)\in BV(\Omega^*)$ such that 
$\lim_{j\rightarrow +\infty}u_j(x)=u_\mu(x)$ in $L^1(\Omega)$. Hence $v(x)-\mu\leq u_\mu(x)\leq v(x)+\mu$ and 
$u_\mu(x)=\psi(x)\text{ outside }\Omega$.  \\
\indent By (2) in Lemma \ref{det:pro}  we have $$A_\mu\leq \AF_\alpha(u_\mu,\Omega^*)\leq \lim_{j\rightarrow +\infty} \AF_\alpha(u_j,\Omega^*)=A_\mu$$
Thus $\AF_\alpha(u_\mu(x),\Omega^*)=A_\mu$.\\
\indent  Let $U_\mu$ be the subgraph of $u_\mu(x)$.  By Theorem \ref{obs:thm}, $U_\mu$ is a minimal set in the closed set 
$$
\{(x,t): x\in \bar{\Omega}, t\in[v(x)-\mu, v(x)+\mu]\}
$$
\indent Set $\Sigma:=\{(x,v(x)+\mu):x\in\Omega\}$. By \eqref{min:eq} $\Sigma$ is minimal. We claim that $\P U_\mu$ can not touch $\Sigma$.\\
\indent  Suppose not. Let $p$ be the common point of $\Sigma$ and $\P U_\mu$. By Theorem \ref{regularity:key:thm}  $\P U_\mu$ is a $C^{1,\alpha}$ graph near $p$.  As a result 
\be  \label{des}
H_{\P U_\mu}\leq 0
\ene 
near $p$ with respect to the outward normal vector of $\P U_\mu$ in the Lipschitz sense.  By Theorem \ref{max:thm:one} and \eqref{des} $\Sigma$ coincides with $\P U_\mu$ near $p$. Define $\Gamma_0=\{(x,v(x)+\mu):x\in\P\Omega\}$. The connectedness implies that $\Gamma_0\subset \P U_\mu$. Note that $\Gamma_0$ is disjoint with $\Gamma$. Again by Theorem \ref{regularity:key:thm} $\P U_\mu$ is a $C^{1,\alpha}$ on $\Gamma_0$.  This means that the tangent space of $\P U_\mu$ is equal to those of $\Sigma$ at $\Gamma_0$. Thus there are tangent vectors of $\P U_\mu$ on each point of $\Gamma$ pointing outward to $\Omega\PLH\R$. This is impossible since $\P U_\mu$ is contained in $\bar{\Omega}\PLH\R$. \\
\indent Consequently $\P U_\mu$ can not touch the boundary $\Sigma:=\{(x,v(x)+\mu):x\in\Omega\}$. With a similar derivation $\P U_\mu$ is also disjoint with the set $\{(x,v(x)-\mu):x\in\Omega\}$. \\
\indent Note that all above derivation is true for any $\mu>\mu_0$. Arguing similarly as in Corollary \ref{de:cor}, by induction we see that $\P U_\mu|_{\Omega\PLH\R}$ is contained in 
         \be \label{eq:key:fact}
         \{(x,t): x\in \bar{\Omega}, t\in[v(x)-\mu_0, v(x)+\mu_0]\}
         \ene 
   This yields $v(x)-\mu_0 \leq u_\mu(x)\leq v(x)+\mu_0$ on $\Omega$ for any $\mu>\mu_0$. Letting $\mu \rightarrow +\infty$ yields that $\{A_\mu\}$ converges to a constant $A$ defined by 
   $$
 A:= \min\{ \AF_\alpha(u,\Omega^*): u\in BV(\Omega^*),  \text{bounded}, u=\psi(x)\text{ outside }\Omega\}
   $$
    By (1) in  Lemma \ref{det:pro} and \eqref{eq:key:fact} we have the estimate 
   \be 
   \{||Du_\mu||_N(\Omega^*)+vol(\Omega^*) \}\leq C
   \ene 
   where $C$ only depends on $\mu_0, v(x)$ and $A$. By the compactness of BV functions,  there is $u_\infty(x)\in BV(\Omega^*)$ such that $\{u_\mu(x)\}$ (possibly a subsequence) converges to $u_\infty(x)$ in $L^1(\Omega)$ as $\mu$ goes to $+\infty$. Thus $u_\infty(x)=\psi(x)\text{ outside }\Omega$. Moreover Lemma \ref{det:pro} gives 
   \be 
   A\leq \AF_\alpha(u_\infty,\Omega^*)\leq \lim_{\mu\rightarrow+\infty} \AF_\alpha(u_\mu,\Omega^*)=A
   \ene 
   Thus $u_\infty(x)$ is the desired bounded BV function for the conclusion (1). The proof is complete. 
\ep  
\section{The NCM property is necessary}\label{morgan}
In this section we give examples to show the NCM property is necessary to obtain the conclusions in Theorem  \ref{thm:A:B} (Theorem \ref{thm:one:two}) and Theorem \ref{thm:MT:C} (Theorem \ref{thm:two:two}).\\
\indent Throughout this section let $S^n$ and $S^n_+$ be the $n$-dimensional sphere and open hemisphere respectively with the standard metric $\sigma_n$. Note that $\P S^n_+$ is a $(n-1)$ dimensional unit sphere and is minimal in $S^n$. Thus by Definition \ref{Def:NCM} $S^n_+$ does not have the NCM property. 
\bt \label{thm:one:two}For any $\alpha\geq n$ there is no solution in $C^2(S^n_+)\cap C(\bar{S^n_+})$ to the Dirichlet problem
\begin{equation}\label{uv:st}
\left\{\begin{split}
div(\frac{Du}{\omega}) &=\F{\alpha}{\omega}\quad x\in \Omega \quad \omega =\sqrt{1+|Du|^2}\\
u(x)&=\psi(x)\quad x\in \partial\Omega
\end{split}\right.  
\end{equation}
for any $\psi(x)\in C(\P\Omega)$. 
\et  
\br  The case of $\alpha=n$ is obtained in appendix C in \cite{Zhou19}. 
\er
\begin{Def} Define  
	\be 
	R^{\alpha+1}_+ :=\{S^n_+\PLH\R, e^{2r\F{\alpha}{n}}(\sigma_n+dr^2) \}
	\ene
\end{Def}
When $\alpha =n$, $R^{\alpha+1}_+$ is isometric to the half upper space in $\R^{n+1}$ via the map
\be 
F: R^{n+1}_+\rightarrow  (x_1,x_2,\cdots, x_{n+1})\subset\R^{n+1}
\ene 
with $F(\theta,\rho)=(x_1,x_2,\cdots,x_{n+1})$ for $x_{n+1}>0$. Here $(\theta,\rho)$ is the polar coordinate of $(x_1,x_2,\cdots, x_{n+1})$ in $\R^{n+1}$. For any $t>0$, consider a family of hypersurfaces $\Sigma_t$ in $R^\alpha_+$ given by 
\be\label{Def:sm}
\Sigma_t:=F^{-1}\{(x_1,x_2,\cdot, x_n,t):(x_1,x_2,\cdot, x_n)\in\R^n\}
\ene 
The following property of $\Sigma_t$ is easily obtained. 
\bl \label{lm:evidence}
For any $t>0$, $\Sigma_{t}$ is unbounded in $R_+^{\alpha+1}$, i.e. there is no bounded inverval $[a,b]$ such that $\Sigma_{t}\subset S_+^n\PLH [a,b]$.
\el 
Now taking the outward normal vector of $\vec{v}_t$ pointing to the positive infinity in $\R^n_+$, we have $\la \vec{v}_t,\P_r\ra >0$ on $\Sigma_t$. Let $H_\alpha$ be the mean curvature of $\Sigma_t$ in $R^\alpha_+$.  Thus $H_n=0$ on $\Sigma_t$. Note that the metric of $\R^{\alpha+1}_+$ can be written as 
$$
e^{2r\F{\alpha-n}{n}}e^{2r}(\sigma_n+dr^2)
$$
By Lemma \ref{lm:mean:curvature} and $H_n=0$,  we have 
$$
H_\alpha =e^{-\F{\alpha-n}{n}r}(\F{\alpha-n}{n}\la v_t,\P_r\ra )
$$
This gives that 
\bl Suppose $\alpha\geq n$.  Then $H_\alpha\geq 0$ on $\Sigma_t$ for each $t>0$ in $\R^{\alpha+1}_+$. 
\el 
Now we are ready to show Theorem \ref{thm:one:two}. 
\bp Now assume there is a $C^2$ solution $u(x)$ to \eqref{uv:st}.  Let $S$ be the graph of $u(x)$ in $\R^\alpha_+$. Note that $\{\Sigma_t\}_{t>0}$ gives a smooth foliation for $\R^\alpha_+$. Then we  set 
   \be 
   t_0=\sup\{t>0: \Sigma_t \cap S\neq \emptyset \}
   \ene 
   Since $F(S)$ is a bounded set in $\R^n_+$. Thus $t_0$ is a finite positive number. Thus $\Sigma_{t_0}$ is tangent to $S$ at some point.  Note that $H_\alpha \geq 0$ on $\Sigma_{t_0}$ with the normal vector pointing to the positive infinity and $S$ is minimal.  By Theorem \ref{max:thm:one} and the connectedness of $\Sigma_{t}$, $ \Sigma_{t}\subset S$.  It contradicts to Lemma \ref{lm:evidence}.  The proof is complete. \ep 
   Let $\psi(x)\in C^1(\P S^n_+)$ and $\Gamma=\{(x,\psi(x)):x\in\P S^n_+\}$. Let $U_\psi$ be the subgraph of $\psi(x)$ in $\P\Omega\PLH\R$, i.e. $\{(x,t):x\in\P\Omega, t<\psi(x)\}$. Define $\fG_\alpha$ as the set of all {\IMC} with compact support in $\bar{R}^{\alpha+1}_+$, i.e. $\bar{S}^n_+\PLH\R$. Here a closed set is compact if it is contained in some set $\bar{S}^n\PLH [a,b]$ where $a<b$ are two finite constants. The area minimizing problem in $R^{n+1}_+$  is to find an {\IMC} $T_0\in \fG_\alpha$ to realize the minimum of 
   \be\label{fin:am:problem}
   \min \{ \bM(T): T\in \fG_\alpha, \P T=\Gamma\}
   \ene 
  \bt\label{thm:two:two}  For any $\alpha\geq n$, no $T\in \fG_\alpha$ can realize the minimum in \eqref{fin:am:problem} in $\R^{\alpha+1}_{+}$. 
  \et 
  \bp Suppose there is a $T_0\in \fG_\alpha$ to realize the minimum in \eqref{fin:am:problem}. Because $T_0$ has a compact support,  we can asssume $T_0\subset \bar{S}^n_+\PLH (a,b)$ for two finite constants $a<b$ and $T_0$ is disjoint with $\bar{S}^n_+\PLH\{a\}$ and $\bar{S}^n_+\PLH\{b\}$. \\
  \indent  By Lemma \ref{eq:lm:st} there is a {\Ca} set $F$ in $S^n\PLH\R$ such that $T_0=[[\P F]]|_{\bar{\Omega}\PLH\R}$. Note that $F$ is a minimal set in $\bar{\Omega}\PLH(a,b)$. \\
  \indent Let $\Sigma_t$ be the smooth  hypersurface in \eqref{Def:sm}. Since $\{\Sigma_t\}_{t>0}$ is a smooth foliation of $R^{\alpha+1}_+$. Define 
  \be 
     t_1=\sup\{t>0: \Sigma_t \cap T_0\neq \emptyset \}
 \ene 
 Since $T_0\in \fG_\alpha$, $t_1$ is a finite number. Thus $\Sigma_t$ is tangent to $T_0$ at some point $p$ in $S^n_+\PLH\R$.  Let $W$ be the domain of $\R^{\alpha+1}_+\backslash \Sigma_{t_0}$ such that $T_0\subset \bar{W}$.\\
 \indent  By Theorem \ref{thm:par:regularity},  then $T_0$ is a $C^{1,\beta}$ graph near $p$ for some $\beta>0$. Since $F$ is a minimal set near $p$, the mean curvature of $T_0$ is equal to $0$ in the Lipschitz sense. By the standard schauder estimate,  $T_0$ is smooth.  On the other hand $H_\alpha\geq 0$ on $\Sigma_{t_1}$ {\wrt} the normal vector pointing to the positive infinity.  By Theorem \ref{max:thm:one} and the connectedness of $\Sigma_{t_1}$, $ \Sigma_{t_1}\subset T_0$.  This yields a contradiction because $\Sigma_{t_1}$ is unbounded in $R_+^{\alpha+1}$ from Lemma \ref{lm:evidence}.  The proof is complete. 
  \ep
\section{Acknowledgement}
The first author was supported by the National Natural Science Foundation of China, Grant No. 11771456. 
The second author was supported by the National Natural Science Foundation of China, Grant No. 11801046, the Foundamental Research Funds for the Central Universities, China, Grant No. 2019 CDXYST0015. 

 \appendix 
\section{Some maximum principles}
In this section $M$ always denotes a smooth Riemannian manifold with a metric $g$ and $\Omega$ is an open $C^{1,\alpha}$ domain in $M$.  Recall that in Definition \ref{def:mc} $div(\vec{v})$ denotes the mean curvature of a smooth hypersurface in $M$ with respect to the normal vector $\vec{v}$. 
\bt\label{max:thm:one} Let $S\subset \bar{\Omega}$ be an orientable connected $C^{1,\alpha}$ hypersurface in $M$. Suppose 
	$H_{S}\leq 0$ and $H_{\P\Omega}\geq 0$ with respect to the outward normal vector of $\Omega$ in the Lipschitz sense. 
If $S$ is tangent to a point $p$ in $\P\Omega$, then $S$ coincides with $\P\Omega$ and is minimal in a neighborhood of $p$. 
\et
\bp Suppose $S$ is tangent to $\P\Omega$ at point $p$. We use the coordinate $(y, t)\in \R^{n-1}\PLH\R$ to denote a normal coordinate near $p=(0,0)$ such that the positive direction of $t$ at $p=(0,0)$ points into the $\Omega$. Thus near $p$ both $S$ and $\P\Omega$ are graphs over a neighborhood of the origin in $\R^{n-1}$. \\
\indent Fix any $C^{1,\alpha}$ surface $\Sigma$ passing through $p$.  In a neighborhood of $p$, $\Sigma$ can be represented as a graph over $W\subset \R^{n-1}$ containing 0.  By the area formula, its area near  can be written as 
\be \label{eq:A}
Area(\Sigma)=\int_{W}  F(y,u,Du)dy \ene 
where $\Sigma=(y,u(y))$ and $u\in C^{1,\alpha}$ and $F=F(y,t,p,t):W\PLH\R \PLH T\R^{n-1}\PLH\R$ is a $C^{\infty}$ map determined by the metric $g$.  \\
\indent With the notation above, there are two functions $u_1(y)$ and $u_2(y)$ such that $S$ and $\P\Omega$ are the graphs of $u_1(y)$ and $u_2(y)$ over $W$ respectively, $u_1(x)\geq u_2(y)$ and $u_1(0)=u_2(0)$. Note that $u_1\in C^{1,\alpha}(W)$ and $u_2(x)\in C^2(W)$. Because $H_S\leq 0$,  for any nonnegative function $\phi(y)\in C^1_0(W)$ we have 
\be \label{eq:B}
\int_{W}(\F{\P F}{\P p_i}(y,u_1, Du_1)\phi_{i}+\F{\P  F}{\P z}(y,u_1,Du_1)\phi )dy \geq 0 
\ene 
Since $H_{\P\Omega}\geq 0$,  then for $\phi(y)\geq 0$ in $C^1_0(W)$ it holds 
\be\label{eq:C}
\int_{W}(\F{\P F}{\P p_i}(y,u_2, Du_2)\phi_{i}+\F{\P  F}{\P z}(y,u_2,Du_2)\phi )dy \leq 0 
\ene
We can assume $Du_1, Du_2$ are uniformly bounded on $W$.  Consider $s(y)=u_1(y)-u_2(y)$. Then $s(y)\geq 0$, $s(0)=0$ with 
\be 
\int_{W}\{a^{ij}s_i\phi_i+b^i s_i\phi_i+cs\}\phi dy \geq 0 
\ene
for any $\phi(y)\geq 0$ in $C_0^1(W)$. Here $(a^{ij})$ is an uniformly positive definite matrix if we take $W$ small enough. By the weak Harnack inequality in Theorem 1.2, \cite{Tru67}, there is a $\gamma >0$ such that 
\be 
\rho^{-\F{n}{\gamma}} ||s||_{\gamma}(K(2\rho))\leq min_{K(\rho)}s=0
\ene 
where $K(\rho)$ is the Eucldiean cube centering at $0\in \R^{n-1}$ with a Euclidean radius $\rho$. Thus $u_1(y)\equiv u_2(y)$ near  a neighborhold of $0$. Thus $S$ coincides with $\P\Omega$ near $p$. Moreover the inequalities in \eqref{eq:B} and \eqref{eq:C} should be two equalities.  Then $\Sigma$ is a critial point of area functional.  This means that $S$ is minimal.  The proof is complete. 
\ep 

\bibliographystyle{abbrv}	
\bibliography{Ref}

\begin{thebibliography}{10}

\bibitem{And82}
M.~T. Anderson.
\newblock Complete minimal varieties in hyperbolic space.
\newblock {\em Invent. Math.}, 69(3):477--494, 1982.

\bibitem{BH16}
J.~Bemelmans and J.~Habermann.
\newblock Surfaces of prescribed mean curvature in a cone.
\newblock {\em Arch. Math. (Basel)}, 107(4):429--444, 2016.

\bibitem{Bourni11}
T.~Bourni.
\newblock {$C^{1,\alpha}$} theory for the prescribed mean curvature equation
  with {D}irichlet data.
\newblock {\em J. Geom. Anal.}, 21(4):982--1035, 2011.

\bibitem{CHH19}
J.-B. Casteras, E.~Heinonen, and I.~Holopainen.
\newblock Dirichlet problem for {$F$}-minimal graphs.
\newblock {\em J. Anal. Math.}, 138(2):917--950, 2019.

\bibitem{CSZ20}
L.~Chen, H.~Shao, and H.~Zhou.
\newblock Uniqueness of minimal surfaces in conformal cones.
\newblock in preparation, 2019.

\bibitem{Ding19}
Q.~Ding.
\newblock Minimal cones and self-expanding solutions for mean curvature flows.
\newblock {\em Mathematische Annalen}, Dec 2019.

\bibitem{DJX16}
Q.~Ding, J.~Jost, and Y.~L. Xin.
\newblock Existence and non-existence of area-minimizing hypersurfaces in
  manifolds of non-negative {R}icci curvature.
\newblock {\em Amer. J. Math.}, 138(2):287--327, 2016.

\bibitem{DS93}
F.~Duzaar and K.~Steffen.
\newblock {$\lambda$} minimizing currents.
\newblock {\em Manuscripta Math.}, 80(4):403--447, 1993.

\bibitem{Fu03}
P.~Fusieger and J.~Ripoll.
\newblock Radial graphs of constant mean curvature and doubly connected minimal
  surfaces with prescribed boundary.
\newblock {\em Ann. Global Anal. Geom.}, 23(4):373--400, 2003.

\bibitem{GT01}
D.~Gilbarg and N.~S. Trudinger.
\newblock {\em Elliptic partial differential equations of second order}.
\newblock Classics in Mathematics. Springer-Verlag, Berlin, 2001.
\newblock Reprint of the 1998 edition.

\bibitem{Giu78}
E.~Giusti.
\newblock On the equation of surfaces of prescribed mean curvature. {E}xistence
  and uniqueness without boundary conditions.
\newblock {\em Invent. Math.}, 46(2):111--137, 1978.

\bibitem{Giu84}
E.~Giusti.
\newblock {\em Minimal surfaces and functions of bounded variation}, volume~80
  of {\em Monographs in Mathematics}.
\newblock Birkh\"{a}user Verlag, Basel, 1984.

\bibitem{GJJ10}
C.~Gui, H.~Jian, and H.~Ju.
\newblock Properties of translating solutions to mean curvature flow.
\newblock {\em Discrete Contin. Dyn. Syst.}, 28(2):441--453, 2010.

\bibitem{Zhou17a}
Z.~Huang, Z.~Zhang, and H.~Zhou.
\newblock Mean curvature flows of closed hypersurfaces in warped product
  manifolds.
\newblock {\em Mathematics Research Letter}, 26(5):1393--1413, 2019.

\bibitem{Hui86}
G.~Huisken.
\newblock Contracting convex hypersurfaces in {R}iemannian manifolds by their
  mean curvature.
\newblock {\em Invent. Math.}, 84(3):463--480, 1986.

\bibitem{Ilm96}
T.~Ilmanen.
\newblock A strong maximum principle for singular minimal hypersurfaces.
\newblock {\em Calc. Var. Partial Differential Equations}, 4(5):443--467, 1996.

\bibitem{LL85}
C.~P. Lau and F.-H. Lin.
\newblock The best {H}\"{o}lder exponent for solutions of the nonparametric
  least area problem.
\newblock {\em Indiana Univ. Math. J.}, 34(4):809--813, 1985.

\bibitem{Law91}
G.~R. Lawlor.
\newblock {\em A sufficient criterion for a cone to be area minimizing}, volume
  446 of {\em Mem. AMS,}.
\newblock 1991.

\bibitem{LY02}
F.~Lin and X.~Yang.
\newblock {\em Geometric measure theory---an introduction}, volume~1 of {\em
  Advanced Mathematics (Beijing/Boston)}.
\newblock Science Press Beijing, Beijing; International Press, Boston, MA,
  2002.

\bibitem{Lin85}
F.-H. Lin.
\newblock {\em {R}egularity for a cloass of parametric obstacle problems
  integrand integral current prescribed mean curvature minimal surface system}.
\newblock ProQuest LLC, Ann Arbor, MI, 1985.
\newblock Thesis (Ph.D.)--University of Minnesota.

\bibitem{LFH89}
F.-H. Lin.
\newblock On the {D}irichlet problem for minimal graphs in hyperbolic space.
\newblock {\em Invent. Math.}, 96(3):593--612, 1989.

\bibitem{LP03}
R.~L\'{o}pez.
\newblock A note on radial graphs with constant mean curvature.
\newblock {\em Manuscripta Math.}, 110(1):45--54, 2003.

\bibitem{LP14}
R.~L\'{o}pez and J.~Pyo.
\newblock Capillary surfaces in a cone.
\newblock {\em J. Geom. Phys.}, 76:256--262, 2014.

\bibitem{Ma18}
L.~Ma.
\newblock Convexity and the {D}irichlet problem of translating mean curvature
  flows.
\newblock {\em Kodai Math. J.}, 41(2):348--358, 2018.

\bibitem{Mor02}
F.~Morgan.
\newblock Area-minimizing surfaces in cones.
\newblock {\em Comm. Anal. Geom.}, 10(5):971--983, 2002.

\bibitem{Rado32}
T.~Rado.
\newblock Contributions to the theory of minimal surfaces.
\newblock {\em Acta Litt. Scient. Univ. Szeged}, 6:1--20, 1932.

\bibitem{Sau16}
F.~Sauvigny.
\newblock Surfaces of prescribed mean curvature {$H(x,y,z)$} with one-to-one
  central projection onto a plane.
\newblock {\em Pacific J. Math.}, 281(2):481--509, 2016.

\bibitem{Sch04}
D.~Schwab.
\newblock Hypersurfaces of prescribed mean curvature in central projection.
  {I}.
\newblock {\em Arch. Math. (Basel)}, 82(3):245--262, 2004.

\bibitem{Sch05}
D.~Schwab.
\newblock Hypersurfaces of prescribed mean curvature in central projection.
  {II}.
\newblock {\em Arch. Math. (Basel)}, 84(2):171--182, 2005.

\bibitem{Sim83}
L.~Simon.
\newblock {\em Lectures on geometric measure theory}, volume~3 of {\em
  Proceedings of the Centre for Mathematical Analysis, Australian National
  University}.
\newblock Australian National University, Centre for Mathematical Analysis,
  Canberra, 1983.

\bibitem{Tam82}
I.~Tamanini.
\newblock Boundaries of {C}accioppoli sets with {H}\"{o}lder-continuous normal
  vector.
\newblock {\em J. Reine Angew. Math.}, 334:27--39, 1982.

\bibitem{Tau80}
E.~Tausch.
\newblock The {$n$}-dimensional least area problem for boundaries on a convex
  cone.
\newblock {\em Arch. Rational Mech. Anal.}, 75(4):407--416, 1980/81.

\bibitem{Tru67}
N.~S. Trudinger.
\newblock On {H}arnack type inequalities and their application to quasilinear
  elliptic equations.
\newblock {\em Comm. Pure Appl. Math.}, 20:721--747, 1967.

\bibitem{Wang98}
X.-J. Wang.
\newblock Interior gradient estimates for mean curvature equations.
\newblock {\em Math. Z.}, 228(1):73--81, 1998.

\bibitem{Wang11}
X.-J. Wang.
\newblock Convex solutions to the mean curvature flow.
\newblock {\em Ann. of Math. (2)}, 173(3):1185--1239, 2011.

\bibitem{Whi09}
B.~White.
\newblock Which ambient spaces admit isoperimetric inequalities for
  submanifolds?
\newblock {\em J. Differential Geom.}, 83(1):213--228, 2009.

\bibitem{Whi15}
B.~White.
\newblock Subsequent singularities in mean-convex mean curvature flow.
\newblock {\em Calc. Var. Partial Differential Equations}, 54(2):1457--1468,
  2015.

\bibitem{Zhang18}
Y.~Zhang.
\newblock On realization of tangent cones of homologically area-minimizing
  compact singular submanifolds.
\newblock {\em J. Differential Geom.}, 109(1):177--188, 2018.

\bibitem{zhou18b}
H.~Zhou.
\newblock Nonparametric mean curvature type flows of graphs with contact angle
  conditions.
\newblock {\em Int. Math. Res. Not. IMRN}, (19):6026--6069, 2018.

\bibitem{ZhouA19}
H.~Zhou.
\newblock The boundary behavior of domains with complete translating, minimal
  and {CMC} graphs in {$N^2\times\Bbb{R}$}.
\newblock {\em Sci. China Math.}, 62(3):585--596, 2019.

\bibitem{Zhou19}
H.~Zhou.
\newblock Generalized solutions to the {D}irichlet problem of translating mean
  curvature equations.
\newblock {\em arXiv:1902.01512}, 2019.

\end{thebibliography}
 \end{document}